\documentclass[11pt]{article}

\usepackage{amsmath}
\usepackage{amssymb}
\usepackage{amsfonts}
\usepackage{amsthm}
\newcommand\blfootnote[1]{%
	\begingroup
	\renewcommand\thefootnote{}\footnote{#1}%
	\addtocounter{footnote}{-1}%
	\endgroup
}

\newtheorem{theorem}{Theorem}[section]
\newtheorem{proposition}[theorem]{Proposition}
\newtheorem{lemma}[theorem]{Lemma}
\newtheorem{corrolary}[theorem]{Corollary}
\theoremstyle{remark}
\newtheorem*{remark}{Remark}
\theoremstyle{definition}
\newtheorem{definition}{Definition}
\usepackage{mathtools}
\numberwithin{equation}{section}
\DeclarePairedDelimiter\abs{\lvert}{\rvert}
\DeclarePairedDelimiter\norm{\lVert}{\rVert}
\newcommand{\normB}[1]{\norm{#1}_{\dot{\mathcal{B}}}}

\newcommand{\normH}[2]{\norm{#2}_{\dot{H}_{#1}}}
\newcommand{\normN}[1]{\norm{#1}_{L_\infty}}
\renewcommand{\S}{Section}
\DeclareMathOperator{\tr}{Tr}
\DeclareMathOperator{\supp}{supp}
\DeclareMathOperator{\conf}{Conf}

\newcommand{\R}{\mathbb{R}}
\newcommand{\N}{\mathbb{N}}
\newcommand{\C}{\mathbb{C}}
\begin{document}
\title{Speed of convergence in the Central Limit Theorem for the determinantal point process with the Bessel kernel}
\date{}
\author{Sergei M. Gorbunov\blfootnote{Moscow Institute of Physics and Technology, Dolgoprudny, Russia}\footnote{Institute for System Programming of the Russian Academy of Sciences, Moscow, Russia}}
\maketitle
\begin{abstract}
	We consider a family of linear operators, diagonalized by the Hankel transform. The Fredholm determinants of these operators, restricted to $L_2[0, R]$, are expressed in a convenient form for asymptotic analysis as $R\to\infty$. The result is an identity, in which the determinant is equal to the leading asymptotic multiplied by an asymptotically small factor, for which an explicit formula is derived. We apply the result to the determinantal point process with the Bessel kernel, calculating the speed of the convergence of additive functionals with respect to the Kolmogorov-Smirnov metric.
\end{abstract}

\section{Introduction}
For $f\in L_\infty(\R_+)\cap L_1(\R_+)$ consider the kernel 
\begin{equation}\label{1:bess_def}
	B_f(x, y) = \int_{\R_+}t\sqrt{xy}J_\nu(xt)J_\nu(yt)f(t)dt,
\end{equation}
where $J_\nu$ is a Bessel function of order $\nu$. Let us fix $\nu> -1$. 
The kernel \eqref{1:bess_def} induces a bounded linear operator on 
$L_2(\R_+)$, which we will also denote by $B_f$. In addition, we let $B_1$ 
be the identity operator. We will refer to the kernel \eqref{1:bess_def} 
and the operator $B_f$ as the Bessel kernel and the Bessel operator.

For any $h\in L_\infty(X)$ we also let $h$ stand for the operator 
of pointwise multiplication on $L_2(X)$. Let $\chi_A(x)$ be the 
characteristic function of the subset $A$. 

In the following work we consider the Fredholm determinant 
$\det(I + \chi_{[0, R]}B_{f}\chi_{[0, R]})$ for $R>0$ and 
sufficiently smooth function $f$.
The determinant gives an exact expression for the Laplace 
transform of additive functionals in the determinantal point 
process with the Bessel kernel (see~\cite{Tracy_1994}). We 
describe the relation between the determinantal point process 
and the operator in \S~5. Briefly, the operator 
$B_{\chi_{[0, 1]}}$ induces a probability measure 
$\mathbb{P}_{J_\nu}$ on countable subsets of $\R_+$ without 
accumulation points. The subsets, called configurations, are 
endowed with certain $\sigma$-algebra (see~\cite{Soshnikov_2000}). 
For any measurable function $b$ on $\R_+$ the additive 
functional is defined to be a measurable function on configurations 
by the formula $S_b(X) = \sum_{x\in X}b(x)$. The probability
 measure induces the random variable $S_b$. For 
 $b\in L_1(\R_+)\cap L_\infty(\R_+)$ we have that the 
 Laplace transform of $S_b$ is expressed by the formula
\[
	\mathbb{E}_{J_\nu}e^{\lambda S_b} = 
	\det(\chi_{[0, 1]}B_{e^{\lambda b}}\chi_{[0, 1]}).
\]

For the function $b$ we consider the additive
functional $S_b^R = S_{b(x/R)}$ as $R\to\infty$. 
For the dilated function we have 
\[
	\det(\chi_{[0, 1]}B_{e^{\lambda b(x/R)}}\chi_{[0, 1]}) 
	= \det(\chi_{[0, R]}B_{e^{\lambda b}}\chi_{[0, R]}).
\]
We conclude that the limit distribution of $S_b^R$ 
can be derived via asymptotic of the determinant.

Let us recall an analogous problem for the sine 
process $\mathbb{P}_{\mathcal{S}}$. Unlike the Bessel
 kernel determinantal point process, it is a measure 
 on configurations on the real line $\R$. The measure 
 is induced by an operator on $L_2(\R)$ with the integral 
 kernel $K_{\mathcal{S}}(x, y) = \frac{\sin \pi(x-y)}{\pi(x-y)}$. 
 The formula for the Laplace transform for 
 $b\in L_1(\R)\cap L_\infty(\R)$ in this case is
\[
	\mathbb{E}_{\mathcal{S}}e^{\lambda S_b^R} = 
	\det(\chi_{[0, 2\pi R]}W_{e^{\lambda b}}\chi_{[0, 2\pi R]}),
\]
where $W_{e^{\lambda b}}$ is a Wiener-Hopf operator with 
symbol $e^{\lambda b}$, which we define in the following section. 
For these determinants we recall two results. First one is the 
Kac-Akhiezer formula \cite[Sect.~10.13]{Bttcher1990}, which 
states that under certain conditions on $b$ we have
\[
	\det(\chi_{[0, 2\pi R]}W_{e^{\lambda b}}\chi_{[0, 2\pi R]}) = 
	\exp(\lambda R c_1^{\mathcal{S}}(b) + \lambda^2 
	c_2^{\mathcal{S}}(b))Q^{\mathcal{S}}_R(\lambda b),
\]
where $ Q^{\mathcal{S}}_R(\lambda b)\to 1$ as $R\to\infty$. See Theorem \ref{2:BO_identity} for the values of $c_1^{\mathcal{S}}(b),
 c_2^{\mathcal{S}}(b)$. We note that the distribution approaches Gaussian 
 if $c^{\mathcal{S}}_1(b)=0$. The second result is an exact formula for 
 the remainder term $Q^{\mathcal{S}}_R(\lambda b)$, see again Theorem 
 \ref{2:BO_identity}. These results are continuous analogues of the 
 Strong Szeg\"{o} Limit Theorem and the Borodin-Okounkov identity for 
 Toeplitz matrices (see~\cite[Sect.~10.4]{Bttcher1990} for the former 
 and~\cite{borodin1999fredholm} for the latter result). Borodin and 
 Okounkov first derived the expression for determinants of Toeplitz 
 matrices via the Gessel theorem and Schur measures. Several different 
 proves under less restrictive assertions were given later 
 \cite{BasorWidom2000}, \cite{boettcher2001determinant}. 
 These approaches may be applied in the continuous case, 
 it was done by Basor and Chen in \cite{BChen2003}. Another proof in
 the continuous case under weaker assumptions was given by Bufetov in \cite{Buf2024}.

Having derived a formula for $Q^{\mathcal{S}}_R(\lambda b)$, it 
is possible to estimate the speed of the convergence of $S_b^R$, 
which was done by Bufetov in \cite[Lemma~3.10]{bufetov2019sineprocess}. 
Theorem \ref{1:KS_estimate} gives a similar estimate for the Bessel 
kernel point process.

An analogue of the Kac-Akhiezer formula in case of the Bessel kernel 
is derived by Basor and Ehrhardt in \cite{Basor_2003}. We state it in Theorem  \ref{5:bessel_asymp}.
Again, the distribution of additive functionals approaches Gaussian 
if $c_1^{\mathcal{B}}(b)=0$. In this paper we will derive an exact 
expression for $Q^{\mathcal{B}}_R(\lambda b)$, therefore proving an 
analogue of the exact identity, and estimate its speed of convergence. 

The result has two special cases $\nu=\pm 1/2$. In these cases 
the Bessel operator \eqref{1:bess_def} is known to be the sum 
of the Wiener-Hopf and Hankel operators:
\[
B_f = W_f \pm H_f, \quad \text{if }\nu=\mp 1/2,
\]
where $H_f$ is a Hankel integral operator with kernel 
$\chi_{[0, \infty)^2}(x, y)\hat{f}(x+y)$. The value of $Q^{\mathcal{B}}_R(b)$ 
for these cases was derived by Basor, Ehrhardt and Widom in \cite{BEW_2003}. 
Our result --- Theorem \ref{1:BO_bessel} --- reproduces their formula. 
It is notable that in these cases $B_f$ has a discrete analogue: a sum of 
Toeplitz and Hankel matrices. An analogue of the Szeg\"{o} theorem for the 
latter has been proved by Johansson \cite{Joh1997}. Furthermore, there is
 a counterpart of the Borodin-Okounkov formula for these matrices, derived 
 by Basor and Ehrhardt in \cite{BE_2008} via operator-theoretic methods. 
 Another proof, similar to the one of Borodin and Okounkov for Toeplitz 
 matrices, has been given by Betea \cite{Betea_2018}, who used symplectic 
 and orthogonal Schur measures.

Lastly, let us recall that an analogous problem of rate of convergence 
of additive functionals has already been considered for random matrix 
ensembles. E. g., classical compact groups were studied by Johansson 
in \cite{Joh1997}. In \cite{KMS2010} authors considered traces of 
random Haar-distributed matrices multiplied by a deterministic one. 
Lambert, Ledoux and Webb studied the speed of the convergence in 
$\beta$-ensembles with respect to the Wasserstein-2 distance in 
\cite{LLW2019}. Gaussian Laguerre and Jacobi ensembles were 
studied by Bufetov and Berezin in \cite{BerBuf2021} via Deift-Zhou
 asymptotic analysis of Riemann-Hilbert problem \cite{RH}. The authors of \cite{Zhig2021} studied convergence of additive functionals in Wigner ensembles. However, 
 the mentioned ensembles are \\finite-dimensional, with number of 
 points growing in the considered limit. On the other hand, the 
 Bessel kernel point process is a measure on infinite configurations. 
 Therefore, we pursue the operator-theoretic approach used in 
 \cite{BChen2003} and \cite{Basor_2003}.

\section{Statement of the result}
For a function $f\in L_1(\R)\cap L_\infty(\R)$ 
define the Fourier transform by the formula
\[
	\hat{f}(\lambda) = \frac{1}{2\pi}\int_\R e^{-i\lambda x}f(x)dx.
\]
Here and subsequently if $f$ is defined on $\R_+$ we extend it evenly. 
In this case the Fourier transform coincides with the cosine transform.
Define the Sobolev $p$-seminorm of a function $f$ on $\R_+$ as follows
\[
	\normH{p}{f}^2 = \int_{\R_+} \abs{\hat{f}(\lambda)}^2
	\abs{\lambda}^{2p}d\lambda.
\]
For an integer $p$ it is equal to $(2\pi)^{-1/2}\norm{f^{(p)}}_{L_2(\R_+)}$ 
by the Parseval theorem. Denote $\norm{f}_{H_p} = \norm{f}_{L_2(\R_+)} + 
\normH{p}{f}$ and define the Sobolev space $H_p(\R_+) = \{f\in L_2(\R_+): 
\norm{f}_{H_p}<+\infty\}$ to be a Banach space with the norm 
$\norm{\cdot}_{H_p}$. The Sobolev space $H_p(\R)$ on the real line 
is defined similarly, we keep the same notation for norms.

Define the following seminorm for a function $f$ on $\R_+$
\[
	\normB{f} = \normH{1}{f} + \normH{3}{f} + 
	\normH{2}{xf(x)} + \normH{3}{x^2f(x)}.
\]
Again denote $\norm{f}_{\mathcal{B}} = \norm{f}_{L_1(\R_+)} + 
\norm{f}_{L_2(\R_+)} + \norm{xf(x)}_{L_\infty(\R_+)} + \normB{f}$ 
and define the space 
\[
\mathcal{B} = \{f\in L_2(\R_+):\norm{f}_{\mathcal{B}}<+\infty\}
\]
to be a Banach space with the norm $\norm{\cdot}_{\mathcal{B}}$.

For a function $f \in L_\infty(\R)$ the Wiener-Hopf operator 
on $L_2(\R_+)$ is defined by the formula
\begin{equation*}
	W_f = \chi_{\R_+}\mathcal{F}f\mathcal{F}^{-1}\chi_{\R_+},
\end{equation*}
where $\mathcal{F}$ is the unitary Fourier transform on $L_2(\R)$: 
$\mathcal{F}h = \sqrt{2\pi}\hat{h}$, $h\in L_2(\R)$. Again, if $f$ is defined 
on $\R_+$, extend it evenly.

Denote 
\[
P_\pm = \mathcal{F}^{-1}\chi_{\R_\pm}\mathcal{F},\quad H_1^\pm(\R) = P_\pm(H_1(\R)),
\]
Then $H_1(\R) = H_1^+(\R)\oplus H_1^-(\R)$. For any function $b\in H_1(\R)$ 
denote $b_\pm = P_\pm b$, $b = b_+ + b_-$. The introduced space $\mathcal{B}$ 
is clearly a subspace of $H_1(\R)$ with embedding given by even continuation, 
so the decomposition is well defined on $\mathcal{B}$, although the components 
may fail to remain in $\mathcal{B}$. We will refer to such decomposition 
as to Wiener-Hopf factorization.

The main result is then stated as follows.
\begin{theorem}\label{1:BO_bessel}
	Let $b \in \mathcal{B}$. We have for any $R>0$
	\begin{equation}\label{1:bessel_ident}
		\det(\chi_{[0, R]}B_{e^{b}}\chi_{[0, R]}) = 
		\exp(Rc_1^{\mathcal{B}}(b) + c_2^{\mathcal{B}}(b) 
		+ c_3^{\mathcal{B}}(b))Q_R^{\mathcal{B}}(b),
	\end{equation}
	where
	\begin{align*}
		&c_1^{\mathcal{B}}(b) = \hat{b}(0),\\
		&c_2^{\mathcal{B}}(b) =  - \frac{\nu}{2}b(0),\\
		&c_3^{\mathcal{B}}(b) = \frac{1}{2}\int_{\mathbb{R}_+}x(\hat{b}(x))^2dx,\\
		&Q^{\mathcal{B}}_R(b) = \det(\chi_{[R, \infty)} 
		W_{e^{b_-}}B_{e^{-b}}W_{e^{b_+}}\chi_{[R, \infty)}).
	\end{align*}
	There exists a constant $C>0$ such that for any 
	$R\ge 1$ and $b\in\mathcal{B}$ the following estimate holds
	\begin{equation}\label{1:bessel_speed}
		\abs{Q^{\mathcal{B}}_R(b) - 1} \le 
		\frac{Ce^{4\normN{b_+}}}{\sqrt{R}}L(b)\exp\left (
		\frac{Ce^{4\normN{b_+}}}{\sqrt{R}}L(b)\right),
	\end{equation}
	where
	\[
	L(b) = (1 + \normN{xb'(x)}^2 + \normB{b}^2)\normB{b}.
	\]
\end{theorem}
\begin{remark}
	Let us again mention cases $\nu=\pm 1/2$. Substituting 
	$B_b = W_b\mp H_b$ into the formula for $Q_R^{\mathcal{B}}(b)$ 
	we obtain that it is equal to
	\begin{equation}\label{1:pmot_sp}
		Q_R^{\mathcal{B}}(b) = \det(I \mp \chi_{[R, \infty)}
		H_{e^{b_--b_+}}\chi_{[R, \infty)}).
	\end{equation}
	We have used the following facts:
	\begin{align*}
		&W_{e^{-b}} = W_{e^{-b_-}}W_{e^{-b_+}}, 
		\quad W_{e^{b_{\pm}}} = e^{W_{b_\pm}}\\
		&W_{e^{b_-}}H_{e^{-b_+}}W_{e^{b_+}} = H_{e^{b_--b_+}}.
	\end{align*}
	For the first two identities see \S~6. 
	For the last see \cite[Section II]{BEW_2003}. 
	Indeed, the expression for remainder \eqref{1:pmot_sp} coincides with 
	the one obtained by Basor, Ehrhardt and Widom in \cite[Formula (3)]{BEW_2003}
\end{remark}

We now proceed to the corollary for the determinanal point process 
with the Bessel kernel $\mathbb{P}_{J_\nu}$ (see \S~5). 
Let $\overline{S_f^R} = S_f^R - \mathbb{E}_{J_\nu}S_f^R$. 
By $F_{R, f}$ and $F_{\mathcal{N}}$ we denote cumulative distribution 
functions of $\overline{S_f^R}$ and the standard Gaussian respectively.
\begin{theorem}\label{1:KS_estimate}
	Let $b\in\mathcal{B}$ be a real-valued function, 
	satisfying $c_3^{\mathcal{B}}(b)=1/2$. Then there exists 
	a constant $C=C\left(L(b), \normN{b_+}\right)$ providing 
	the following estimate for the Kolmogorov-Smirnov distance for any $R\ge 1$
	\begin{equation}
		\sup_x\abs{F_{R, b}(x) - F_{\mathcal{N}}(x)} \le \frac{C}{\ln R}.
	\end{equation}
\end{theorem}
\newpage
\section{Structure of the paper}
Let us first recall the result of Basor and Chen.

\begin{theorem}\label{2:BO_identity}
	For $f\in H_1(\R)\cap L_1(\R)$ we have
	\[
		\det(\chi_{[0, R]}W_{e^f}\chi_{[0, R]}) 
		= \exp(Rc_1^{\mathcal{S}}(f) + c_2^{\mathcal{S}}(f))Q_R^{\mathcal{S}}(f),
	\]
	where
	\begin{align*}
		&c_1^{\mathcal{S}}(f) = \hat{f}(0),\\
		&c_2^{\mathcal{S}}(f) = \frac{1}{2}\int_{\R_+} 
		x\hat{f}(x)\hat{f}(-x)dx,\\
		&Q_R^{\mathcal{S}}(f) = \det(\chi_{[R, \infty)}
		W_{e^{f_-}}W_{e^{-f_+}}W_{e^{-f_-}}W_{e^{f_+}}\chi_{[R, \infty)}).
	\end{align*}
\end{theorem}
\begin{remark}
	We state Theorem \ref{2:BO_identity} under different conditions 
	than in \cite{BChen2003} for a clearer outline of proof 
	of Theorem \ref{1:BO_bessel}.
\end{remark}
The proof consists of three main steps.

\textbf{Step 1.} Wiener-Hopf factorization properties give the following identity
\begin{equation}\label{2:BO_first}
	\det(\chi_{[0, R]}W_{e^f}\chi_{[0, R]}) = e^{Rc_1^{\mathcal{S}}(f)}
	\det(\chi_{[0, R]}W_{e^{-f_+}}W_{e^f}W_{e^{-f_-}}\chi_{[0, R]}).
\end{equation}

\textbf{Step 2.} It will be shown that the operator 
$W_{e^{-f_+}}W_{e^f}W_{e^{-f_-}}$ is of determinant class. 
Apply the Jacobi-Dodgson identity: $\det(PAP) = \det(A)\det(QA^{-1}Q)$ 
for orthogonal projections $P, Q$, satisfying $P+Q=I$, and determinant 
class invertible $A$. We have
\begin{equation}\label{2:BO_second}
	\frac{\det(\chi_{[0, R]}W_{e^{-f_+}}W_{e^f}W_{e^{-f_-}}
		\chi_{[0, R]})}{ \det(\chi_{[R, \infty)}W_{e^{f_-}}
		W_{e^{-f_+}}W_{e^{-f_-}}W_{e^{f_+}}\chi_{[R, \infty)})}= 
		\det(W_{e^{-f_+}}W_{e^f}W_{e^{-f_-}}),
\end{equation}
where $W_{e^f}^{-1} = W_{e^{-f_+}}W_{e^{-f_-}}$, as follows 
from the properties of the Wiener-Hopf factorization.

\textbf{Step 3.} From properties of the Fredholm determinant 
and Wiener-Hopf factorization we have $\det(W_{e^{-f_+}}W_{e^f}W_{e^{-f_-}}) =
 \det(W_{e^{f_-}}W_{e^{f_+}}W_{e^{-f_-}}W_{e^{-f_+}})$. 
 The Widom formula (see \cite{Widom1982-qq}) states that 
 the last determinant is well-defined and its value is
\begin{equation}\label{2:BO_last}
	\det(W_{e^{f_-}}W_{e^{f_+}}W_{e^{-f_-}}W_{e^{-f_+}}) = e^{c_2^{\mathcal{S}}(f)}.
\end{equation}
This finishes the proof.

We now pass to the proof of our main result, Theorem \ref{1:BO_bessel}.
\begin{lemma}\label{2:bessel_div}
	For $b\in H_1(\R_+)\cap L_1(\R_+)$ we have
	\begin{equation}\label{2:bessel_div_formula}
		\det(\chi_{[0, R]}B_{e^b}\chi_{[0, R]}) =
		 e^{Rc_1^{\mathcal{B}}(b)}\det(\chi_{[0, R]}W_{e^{-b_+}}
		 B_{e^{b}}W_{e^{-b_-}}\chi_{[0, R]}).
	\end{equation}
\end{lemma}

It is not known if $W_{e^{-b_+}}B_{e^{b}}W_{e^{-b_-}}-I$ is trace class, 
so the Jacobi-Dodgson identity cannot be applied directly. 
However, one can observe that the equality \eqref{2:BO_second} 
means that the ratio of the determinants in the left-hand side 
does not depend on $R$. In our case we can use an analogue of 
the Jacobi-Dodgson identity (Corollary \ref{5:JD_ext}) to show 
the independence of the ratio from $R$.

\begin{lemma}\label{2:det_relation}
	Let $b\in H_1(\R_+)\cap L_1(\R_+)$ and $\chi_{[R, \infty)}(B_{e^{-b}} - 
	W_{e^{-b}})\chi_{[R, \infty)}$ be a trace class operator for any $R>0$. 
	Then there exists $Z(b)$ such that for any $R>0$ we have
	\begin{equation}\label{2:det_relation_eq}
		\frac{\det(\chi_{[0, R]}W_{e^{-b_+}}B_{e^{b}}W_{e^{-b_-}}\chi_{[0, R]})}
		{\det(\chi_{[R, \infty)}W_{e^{b_-}}B_{e^{-b}}W_{e^{b_+}}\chi_{[R, \infty)})} = Z(b).
	\end{equation}
\end{lemma}

Observe that the denominator in \eqref{2:det_relation_eq} 
equals $Q_R^{\mathcal{B}}(b)$. The central step in the proof of 
Theorem \ref{1:BO_bessel} is the following.

\begin{lemma}\label{2:bess_diff_est}
	For $b\in L_1(\R_+)\cap L_\infty(\R_+)$ such that $\normB{b}< \infty$ 
	and $z\in\C$ we have that $\chi_{[R, \infty)}(B_{b+z} - 
	W_{b+z})\chi_{[R, \infty)}$ is trace class for any $R>0$. 
	There exists a constant $C$ such that for any $R\ge1$, $z\in\C$ 
	and $b$ satisfying conditions above we have 
	\begin{equation}
		\norm{\chi_{[R, \infty)}(B_{b+z} - W_{b+z})
			\chi_{[R, \infty)}}_{\mathcal{J}_1} \le \frac{C}{\sqrt{R}}\normB{b}.
	\end{equation}
\end{lemma}
The operator $Q_R^{\mathcal{B}}(b)$ is related to the operator 
in \eqref{2:bess_diff_est} by the formula \eqref{5:denominator_expr} below. 
This allows to obtain the estimate \eqref{1:bessel_speed}. Further, the 
asymptotic for the numerator in \eqref{2:det_relation_eq} follows from 
Lemma \ref{2:bessel_div} and the asymptotic result of Basor and Ehrhardt 
(see Theorem \ref{5:bessel_asymp}). This proves that 
$Z(b) = \exp(c_2^{\mathcal{B}}(b) + c_3^{\mathcal{B}}(b))$.

The rest of the paper has the following structure. 
In \S~4 we recall some facts and notation for 
trace class and Hilbert-Schmidt operators. 
Then in \S~5 we explain relation between the Fredholm 
determinant in Theorem \ref{1:BO_bessel} and determinantal 
point process with the Bessel kernel. In \S~6 we prove 
properties of the Wiener-Hopf factorization and deduce Lemma 
\ref{2:bessel_div}. \S~7 presents the proof of Lemma 
\ref{2:bess_diff_est}. We conclude the proof of Theorems 
\ref{1:BO_bessel} and \ref{1:KS_estimate} in \S~8.

\section{Trace class and Hilbert-Schmidt operators}
In this section we recall certain theorems for symmetrically-normed 
ideals. Here and below by $\mathcal{J}_1$ and $\mathcal{J}_2$ we 
denote the ideals of trace class and Hilbert-Schmidt operators 
respectively. For the basic definitions we refer the reader to 
\cite{Simon2010} and \cite{Simon2015}. Let $E$ be $\R$ or $\R_+ = [0, \infty)$.
\begin{definition}
	The operator $K$ on $L_2(E)$ is locally trace class 
	if for any bounded measurable subset $B\subset E$ 
	the operator $\chi_B K \chi_B$ is trace class. 
	Let $\mathcal{J}_1^{loc}(L_2(E))$ stand for the 
	space of locally trace class operators.
\end{definition}

Recall that for $K\in \mathcal{J}_1$ the Fredholm 
determinant $\det (I+K)$ is well defined. It is 
continuous with respect to the trace norm. We also 
define the regularized determinant as follows
\[
	\det{}_2 (I + K) = \exp(-\tr(K))\det(I+K).
\]
The introduced function is continuous with 
respect to the Hilbert-Schmidt norm. It is then 
extended by continuity to all Hilbert-Schmidt operators.

The following theorem implies that for a function 
$f\in L_1(\R_+)\cap L_\infty(\R_+)$ the determinants 
$\det (I + \chi_{[0, R]}B_f\chi_{[0, R]})$, $\det 
(I + \chi_{[0, R]}W_f\chi_{[0, R]})$ are well defined.
\begin{theorem}[
	{\cite[Theorem~3.11.9]{Simon2015}}]\label{3_0:mercer_th}
	Let $K$ be a continuous kernel on $[a, b]^2$ 
	and induce a positive self-adjoint operator 
	$K$ on $L_2[a, b]$. Then the operator $K$ is 
	trace-class. In addition, for any trace class 
	$K$ with continuous kernel we have
	\begin{equation}\label{3_0:diagonal_tr}
		\tr(K) = \int_a^bK(x, x)dx.
	\end{equation}
\end{theorem}
We will refer to formula \eqref{3_0:diagonal_tr} 
as to calculation of trace over the diagonal. 

Recall that the kernel of an integral operator 
is defined almost everywhere on $[a, b]^2$. Thus, 
if we change a kernel, satisfying conditions of Theorem 
\ref{3_0:mercer_th} on the diagonal, the formula 
\eqref{3_0:diagonal_tr} may fail. This explains the 
continuity requirement. Let us now extend the class 
of kernels, for which the formula \eqref{3_0:diagonal_tr} holds.

Let $K$ be a self-adjoint trace class operator. 
By the Hilbert-Schmidt theorem we have the spectral 
decomposition, that is its kernel can be chosen in the form
\begin{equation}\label{3_0_eq:spec_decomp}
	K(x, y) = \sum_{i\in\N}\lambda_i\phi_i(x)\phi^*_i(y),
\end{equation}
where $\{\phi_i\}_{i\in\N}$ forms an orthonormal 
basis. Now by definition the trace is equal to
\[
	\tr(K) = \int_a^b \sum_{i\in\N}\lambda_i\phi_i(x)\phi^*_i(x)dx,
\]
which is equal to $\int_a^b K(x, x)dx$ if we 
choose $K$ as in \eqref{3_0_eq:spec_decomp}. Observe 
that this does not require the kernel to be continuous, 
but the choice of the kernel holds up to values on 
subsets of $[a, b]^2$ having zero measure projections. 
Also observe that for a function $f\in L_\infty[a, b]$ 
the kernel of the operator $fK$ can be chosen as
\[
	(fK)(x, y) = \sum_{i\in\N}\lambda_i (f(x)\phi_i(x))\phi_i^*(y)
\]
and its trace equals $\int_a^b f(x)K(x, x)$. We will 
use these observations in the following section.

We will also employ the following theorem.
\begin{theorem}[{\cite[Theorem~2.2]{article}}]\label{3_0:HH_th}
	For $A, B \in \mathfrak{B}(\mathcal{H})$ having 
	a trace class commutator we have that $e^Ae^Be^{-A-B}- I$ 
	is trace class. For the determinant we have
	\begin{equation}
		\ln\det (e^{A}e^{B}e^{-A-B}) = \frac{1}{2}\tr([A, B]).
	\end{equation}
\end{theorem}

\section{Application to the determinantal point process with the Bessel kernel}
Let $E$ be $\R$ or $\R_+ = [0, \infty)$.
Recall that a configuration on $E$ is a not more than 
countable subset of $E$ without accumulation points. We 
denote the set of configurations by $\conf(E)$. 
It is endowed with certain $\sigma$-algebra of measurable 
subsets $\mathfrak{X}$ (see \cite{Soshnikov_2000}). A point 
process $\mathbb{P}$ is defined to be a probability measure 
on $(\conf(E), \mathfrak{X})$.

For a configuration $X\in \conf(E)$ and a 
measurable function $f:E\to\C$ define additive 
and multiplicative functionals by the following 
formulae respectively
\[
	S_f(X) = \sum_{x\in X}f(x),
\]
\[
	\Psi_{1+f}(X) = \prod_{x\in X}(1+f(x)).
\]
Point processes induce respective random variables.
\begin{definition}
	A point process $\mathbb{P}_K$ on $\conf(E)$ 
	is determinantal if there exists an operator 
	$K\in \mathcal{J}_1^{loc}(L_2(E))$, satisfying 
	the following condition for any bounded measurable 
	function $f$ with compact support $B=\supp f$
	\begin{equation}\label{6:dpp_def_eq}
		\mathbb{E}_K \Psi_{1+f} = \det (I + fK\chi_B).
	\end{equation}
\end{definition}

We consider the limit distribution of $S_f^R = S_{f(x/R)}$ 
as $R\to\infty$ for the determinantal point process with 
the Bessel kernel $\mathbb{P}_{J_\nu}$ \cite{Tracy_1994} 
on $\R_+$, induced by operator $B_{\chi_{[0, 1]}}$.

Theorem \ref{2:BO_identity} describes the considered 
limit for the sine process $\mathbb{P}_{\mathcal{S}}$ 
on $\R$, induced by operator $\mathcal{F}^{-1}\chi_{[-\pi, \pi]}
\mathcal{F}$ on $L_2(\R)$. To be precise, we formulate the following statement.

\begin{proposition}\label{6:lap_det}
	We have for $f\in L_1(E)\cap L_\infty(E)$
	\begin{equation}
		\mathbb{E}_{J_\nu}\Psi^R_{1+f} = 
		\det(I + \chi_{[0, R]}B_f\chi_{[0, R]}),
	\end{equation}
	\begin{equation}
		\mathbb{E}_{\mathcal{S}}\Psi_{1+f}^R = 
		\det(I + \chi_{[0, 2\pi R]}W_f\chi_{[0, 2\pi R]}).
	\end{equation}
\end{proposition}

The Laplace transform of additive 
functionals is expressed via multiplicative 
functionals by the following formula
\[
	\mathbb{E}e^{\lambda S_b^R} = \mathbb{E}\Psi^R_{e^{\lambda b}}.
\]
Therefore Proposition \ref{6:lap_det} implies that it 
is equal to the Fredholm determinant $\det(\chi_{[0, R]}
B_{e^{\lambda b}}\chi_{[0, R]})$ for the Bessel kernel 
determinantal point process and $\det(\chi_{[0, 2\pi R]}
W_{e^{\lambda b}}\chi_{[0, 2\pi R]})$ for the sine process.

In order to prove Proposition \ref{6:lap_det} we follow Bufetov 
\cite[Section~2.9]{bufetov2019sineprocess} and extend the 
formula \eqref{6:dpp_def_eq}. Let $K$ be a Hilbert-Schmidt 
self-adjoint operator on $L_2(E)$. Choose the kernel $K$ 
to coincide with its spectral decomposition (see \S~4). 
Define the extended Fredholm determinant as follows
\begin{equation}\label{6:reg_det_def}
	\det (I+K) = \exp\left( \int_E K(x, x)dx \right)\det{}_2(I+K)
\end{equation}
if the integral of the diagonal exists. Observe, that 
the definition does not requre $K$ to be trace class, 
though it coincides in this case, so we keep the same notation.

Let $\Pi$ be a locally trace class orthogonal projectior 
on $L_2(E)$. By the Macchi-Soshnikov \cite{Macchi1975}, 
\cite{Soshnikov_2000} theorem $\Pi$ defines a determinantal 
measure $\mathbb{P}_\Pi$. Choose the kernel $\Pi(x, y)$ to 
coincide with the spectral decomposition of $\chi_B \Pi \chi_B$ 
for every compact $B\subset E$. Introduce the measure 
$d\mu_\Pi(x) = \Pi(x, x)dx$ on $E$.
\begin{proposition}
	If $f\in L_1(E, d\mu_\Pi)\cap L_\infty(E, d\mu_\Pi)$ 
	we have that $\Psi_{1+f}\in L_1(\conf(E), 
	\mathbb{P}_\Pi)$ and
	\begin{equation}\label{6:reg_expr_mf}
		\mathbb{E}_\Pi\Psi_{1+f} = \det(I+f\Pi).
	\end{equation}
	In other words, we have that multiplicative functionals 
	define a continuous mapping 
	\[
	\Psi_{1+(-)}: L_\infty(E, d\mu_\Pi)\cap 
	L_1(E, d\mu_\Pi)\to L_1(\conf(E), \mathbb{P}_\Pi).
	\]
\end{proposition}
\begin{proof}
	The formula holds for $f\in L_\infty(E)$ with 
	compact support due to the choice of the kernel 
	$\Pi(x, y)$. Express the extended determinant in 
	the statement by the definition to obtain two factors. 
	The exponent of the integral of the diagonal is 
	continuous with respect to $L_1(E, d\mu_\Pi)$ norm. 
	The regularized determinant is continuous with respect 
	to the Hilbert-Schmidt norm, which is equal to
	\[
		\norm{f\Pi}_{\mathcal{J}_2}^2 = \int_E\abs{f(x)\Pi(x, y)}^2dxdy 
		= \int_E\abs{f(x)}^2\Pi(x, y)\Pi(y, x)dxdy = \int_E\abs{f(x)}^2d\mu_\Pi(x)
	\]
	and is continuous with respect to $L_2(E, d\mu_\Pi)$ norm. 
	Therefore multiplicative functionals define a continuous 
	mapping on the dense subset of $L_1(E, d\mu_\Pi)\cap L_\infty(E, d\mu_\Pi)$ 
	 and, therefore, on all space. The extention of the continuous mapping coincides with the expectation of multiplicative functionals by the Beppo Levi Theorem, since one can monotonously approximate any multiplicative functional $\Psi_{1+\abs{f}}$, $f\in L_1(E, d\mu_\Pi)\cap L_\infty(E, d\mu_\Pi)$ by $\Psi_{1+\abs{f_n}}$, $\abs{f_n}=\chi_{[-n, n]}\abs{f}$. The formula \eqref{6:reg_expr_mf} 
	is extention by continuity of the formula \eqref{6:dpp_def_eq}.
\end{proof}
\begin{proof}[Proof of Proposition \ref{6:lap_det}]
	The proof for the sine process case may be found 
	in \cite[Lemma~3.3]{bufetov2019sineprocess}. 
	We proceed differently by proving the equality 
	between the extended determinants of $I + 
	f\mathcal{F}^{-1}\chi_{[-\pi, \pi]}\mathcal{F}$, 
	$I + fB_{\chi_{[0, 1]}}$ and the Fredholm determinants 
	of $I + \chi_{[0, 2\pi]}W_f\chi_{[0, 2\pi]}$, $I + 
	\chi_{[0, 1]}B_f\chi_{[0, 1]}$ respectively. Observe, 
	that our choice of kernels for the considered processes 
	coincides with the spectral decomposition by Theorem \ref{3_0:mercer_th}. 
	It is straightforward in both cases, that the integrals of the diagonals coincide.
	
	The regularized determinant is unitarily invariant. 
	Recall that $B_{\chi_{[0,1]}}$ is diagonalized by 
	the Hankel transfom $H_\nu$ (see \S~8 and 
	formula \eqref{5:bess_def_inf}). We use formula 
	\eqref{5:bess_def_inf} in order to establish 
	the following unitary equivalence
	\[
		fH_\nu \chi_{[0, 1]}H_\nu \sim H_\nu f H_\nu\chi_{[0, 1]},
	\]
	where the last operator has the same regularized 
	determinant as $\chi_{[0, 1]}B_f\chi_{[0, 1]}$. 
	Notably, the obtained operator is trace class 
	and the Fredholm determinant is well defined.
	
	For the sine kernel use the translation 
	operator $T_tf(x) = f(x+t)$ to proceed as follows
	\[
		\mathcal{F}^*\chi_{[-\pi, \pi]}\mathcal{F}f = 
		\mathcal{F}^*T_{-\pi}\chi_{[0, 2\pi]}T_{\pi}
		\mathcal{F}f = e^{-i\pi x}\mathcal{F}^* 
		\chi_{[0, 2\pi]}\mathcal{F}fe^{i\pi x}
		\sim \mathcal{F}^* \chi_{[0, 2\pi]}\mathcal{F}f
	\]
	to obtain a unitary equivalence. The Fourier 
	transform gives the following unitary equivalence
	\[
		\mathcal{F}^* \chi_{[0, 2\pi]}\mathcal{F}f
		\sim \chi_{[0, 2\pi]}\mathcal{F}f\mathcal{F}^*,
	\]
	where the determinant of the obtained operator 
	coincides with the one for $\chi_{[0, 2\pi]}W_f\chi_{[0, 2\pi]}$. 
	Again, the obtained operator is trace class. Lastly, it is 
	left to observe that the Fourier and Hankel transforms 
	commute with the dilation unitary operator $U_R f(x) = R^{-1/2}f(x/R)$.
\end{proof}

\section{Wiener-Hopf factorization on $H_1(\R)$}

Recall that in general Wiener-Hopf operator as a 
mapping $\mathbf{W}:f\mapsto W_f$, $L_\infty(\R)\to
\mathfrak{B}(L_2(\R_+))$ is not a Banach algebra 
homomorphism. But it can be restricted to certain 
subalgebras, on which it preserves multiplication.

The space $H_1(\R)$ is a Banach algebra. 
The spaces
\[
H_1^\pm(\R) = \{f\in H_1(\R):\supp\hat{f}\subset\R_\pm\}
\]
are Banach subalgebras since $\supp(\hat{f}*\hat{g}) 
\subset \supp\hat{f} + \supp\hat{g}$. 
It is clear that operators $P_\pm = \mathcal{F}^{-1}
\chi_{\R_\pm}\mathcal{F}$ map $H_1(\R)$ into itself 
and $P_\pm H_1(\R)=H_1^\pm(\R)$.
Then $\mathbf{W}$ restricted to $H_1^\pm(\R)$ is a 
Banach algebra homomorphism. As before, we consider 
$H_1(\R_+)$ to be embedded into $H_1(\R)$ by even 
continuation, which defines the decomposition for 
functions on $\R_+$.

We also let $H_1^\pm(\R)\oplus \C = \{f+c:f\in H_1^\pm(\R), 
c\in\C\}$ be the Banach algebras $H_1^\pm(\R)$ with adjoined unit. 
In these algebras an element $\Phi(b)\in H_1^\pm(\R)\oplus\C$ is 
well defined for $b\in H_1^\pm\oplus\C$ and entire $\Phi$ as a 
converging in norm power series from $b$.
The following statement shows that Wiener-Hopf operators 
$W_f$ have kernels for $f\in H_1(\R)$.
\begin{lemma}[{\cite[Proposition 5.2]{Basor_2003}}]
	Let $a$ be a function on $\R$ such that 
	$\norm{a}_{H_{1/2}} + \norm{\hat{a}}_{L_1}<\infty$. 
	Then $\mathcal{F}^{-1}a\mathcal{F}$ is an integral 
	operator on $L_2(\R)$ with the kernel $k(x, y) = \hat{a}(x-y)$.
	\label{3:WH_kernels}
\end{lemma}

\begin{proposition}\label{3:WH_properties}
	
	1. The map $\boldsymbol{W}\big|_{H_1^\pm(\R)}: f \mapsto W_f$ 
	defines homomorphisms of the Banach algebras 
	$H_1^\pm(\R)\oplus \C\to \mathfrak{B}(L_2(\R_+))$.
	
	2. We have for $b_{\pm}\in H_1^\pm(\R)\oplus \C$
	\begin{align*}
		&\chi_{[0, R]}W_{b_+} = 
		\chi_{[0, R]}W_{b_+}\chi_{[0, R]} 
		&W_{b_+}\chi_{[R, \infty)} = 
		\chi_{[R, \infty)}W_{b_+}\chi_{[R, \infty)}\\
		&W_{b_-}\chi_{[0, R]} = 
		\chi_{[0, R]}W_{b_-}\chi_{[0, R]}
		 &\chi_{[R, \infty)}W_{b_-} = 
		 \chi_{[R, \infty)}W_{b_-}\chi_{[R, \infty)}.
	\end{align*}
	
	3. We have that $W_{b_-}W_{b_+}=W_{b_-b_+}$ 
	for $b_\pm\in H_1^\pm(\R)\oplus\C$.
	
	In particular, we have $W_{e^{b_\pm}} = e^{W_{b_\pm}}$ 
	and $W_{e^b} = e^{W_{b_-}}e^{W_{b_+}}$.
\end{proposition}
\begin{proof}
	It is enough to prove the statements for 
	$b_\pm\in H_1^\pm(\R)$. We give proofs 
	for $H_1^+(\R)$, the case of 
	$H_1^-(\R)$ is proved similarly.
	
	1. By Lemma \ref{3:WH_kernels} the Wiener-Hopf 
	operators have kernels. Obviously, we have 
	$\hat{b}_\pm(x) = \chi_{\R_\pm}(x)\hat{b}_\pm(x)$. 
	Let $a, b$ be functions from $H_1^+(\R)$.
	The kernel of $W_{ab}$ equals
	\begin{multline*}
		W_{ab}(x, y) = \chi_{\R_+^2}(x, y)\int_\R 
		\hat{a}(x-y-t)\chi_{\R_+}(t)\hat{b}(t)\chi_{\R_\pm}(t)dt =\\
		= \chi_{\R_+^2}(x, y)\int_\R\hat{a}(x-t)
		\chi_{\R_+}(t-y)\hat{b}(t-y)dt = \\
		= \chi_{\R_+^2}(x, y)\int_\R\hat{a}(x-t)
		\chi_{\R_+}(t)\hat{b}(t-y)dt = (W_a W_b)(x, y),
	\end{multline*}
	where the identity $\chi_{\R_+}(y)\chi_{\R_+}(t-y) = 
	\chi_{\R_+}(y)\chi_{\R_+}(t-y)\chi_{\R_+}(t)$ was used.
	
	2. Since $\chi_{[0, R]}(x)\chi_{\R_+}(x-y) = 
	\chi_{[0, R]}(x)\chi_{\R_+}(x-y)\chi_{[0, R]}(y)$ 
	for all $(x, y) \in \R_+$, one can write for the 
	kernel of $(\chi_{[0, R]}W_{b_+})(x, y)$
	\begin{multline}
		(\chi_{[0, R]}W_{b_+})(x, y) = \chi_{[0, R]}(x)
		\hat{b}_+(x-y)\chi_{\R_+}(y) =\\= \chi_{[0, R]^2}(x, y)\hat{b}_+(x-y)
		= (\chi_{[0, R]}W_{b_+}\chi_{[0, R]})(x, y).
	\end{multline}
	The proof for $W_{b_+}\chi_{[R, +\infty)}$ similarly follows from 
	\[
	\chi_{[R, \infty)}(y)\chi_{\R_+}(x-y)\chi_{\R_+}(x) = 
	\chi_{[R, \infty)}(y)\chi_{\R_+}(x-y)\chi_{[R, \infty)}(x).
	\]
	
	3. Again write the kernel
	\[
		W_{b_-b_+}(x, y) = \chi_{\R_+^2}(x, y)
		\int_\R\hat{b}_-(x - y - t)\hat{b}_+(t)\chi_{\R_+}(t)dt.
	\]
	As for the first part, the statement follows from 
	$\chi_{\R_+}(y)\chi_{\R_+}(t-y) = \chi_{\R_+}(y)
	\chi_{\R_+}(t-y)\chi_{\R_+}(t)$.
\end{proof}

\begin{proof}[Proof of Lemma \ref{2:bessel_div}]
	Properties of Fredholm determinants and the 
	first and second statements of Proposition 
	\ref{3:WH_properties} yield the following expression
	\begin{multline*}
		\det (\chi_{[0, R]}B_{e^b}\chi_{[0, R]}) = 
		\det(\chi_{[0, R]}W_{e^{b_+}}W_{e^{-b_+}}B_{e^b}
		W_{e^{-b_-}}W_{e^{b_-}}\chi_{[0, R]}) =\\
		= \det (\chi_{[0, R]}W_{e^{-b_+}}B_{e^b}
		W_{e^{-b_-}}\chi_{[0, R]})\det (e^{\chi_{[0, R]}W_{b_-}
			\chi_{[0, R]}}e^{\chi_{[0, R]}W_{b_+}\chi_{[0, R]}}).
	\end{multline*}
	It remains to prove that
	\begin{equation}\label{3:div_expr}
		\det (e^{\chi_{[0, R]}W_{b_-}\chi_{[0, R]}}
		e^{\chi_{[0, R]}W_{b_+}\chi_{[0, R]}}) = e^{R\hat{b}(0)}.
	\end{equation}
	By Theorem \ref{3_0:mercer_th} the operator 
	$\chi_{[0, R]}W_b\chi_{[0, R]}$ is trace class for 
	$b\in L_1(\R_+)\cap L_\infty(\R_+)$, so the left-hand 
	side of \eqref{3:div_expr} is equal to
	\[
		\det (e^{\chi_{[0, R]}W_{b_-}\chi_{[0, R]}}
		e^{\chi_{[0, R]}W_{b_+}\chi_{[0, R]}}e^{-\chi_{[0, R]}
			W_b\chi_{[0, R]}})e^{\tr(\chi_{[0, R]}W_b\chi_{[0, R]})},
	\]
	where $\tr(\chi_{[0, R]}W_b\chi_{[0, R]}) = 
	R\hat{b}(0)$ is calculated over the diagonal.
	A direct calculation shows that the operator 
	$\chi_{[0, R]}W_f\chi_{[0, R]}$ is Hilbert-Schmidt 
	for $f\in H_1(\R_+)$. Hence the commutator $[\chi_{[0, R]}
	W_{b_-}\chi_{[0, R]}, \chi_{[0,R]}W_{b_+}\chi_{[0, R]}]$ 
	is trace class and its trace is equal to zero. 
	Theorem \ref{3_0:HH_th} gives
	\[
		\det (e^{\chi_{[0, R]}W_{b_-}\chi_{[0, R]}}
		e^{\chi_{[0, R]}W_{b_+}\chi_{[0, R]}}
		e^{-\chi_{[0, R]}W_b\chi_{[0, R]}}) = 1
	\]
	This finishes the proof of \eqref{3:div_expr}.
\end{proof}

\section{Proof of Lemma \ref{2:bess_diff_est}}
In this section Lemma \ref{2:bess_diff_est} 
is established. Our calculations follow Basor 
and Ehrhardt \cite[Lemmata~2.6, 2.7, 2.8]{Basor_2003}. The section goes as follows. In Propositions \ref{4:est_separable}, \ref{4:trunc_trace} and Corollary \ref{4:trunc_corr} we establish estimates on trace norms of integral operators with kernels of certain forms. Then we recall necessary properties of Bessel functions. Lastly, we proceed to the proof of Lemma \ref{2:bess_diff_est} after proving properties of functions in the introduced space $\mathcal{B}$ in Lemma \ref{4:B_space_prop}.
\begin{proposition}\label{4:est_separable}
	Let $K$ be an integral operator on 
	$L_2[R, \infty)$ with the following kernel
	\[
		K(x, y) = \int_{\R_+} a(t)h_1(x, t)h_2(y, t)dt,
	\]
	where $h_1, h_2, a$ are some measurable functions.
	Then there is an estimate
	\begin{equation}\label{4:est_separable_eq}
		\norm{K}_{\mathcal{J}_1}\le\int_0^\infty 
		\abs{a(t)}\left(\int_R^\infty \abs{h_1(x, t)}^2
		dx\right )^{1/2}\left (\int_R^\infty 
		\abs{h_2(y, t)}^2dy\right)^{1/2}dt
	\end{equation}
	if the right-hand side is finite.
\end{proposition}
\begin{proof}
	Let $f, g\in L_2[R, \infty)$. Denote 
	$h^t_i(x) = h_i(x, t)$, $i=1, 2$. Then we have
	\begin{multline}\label{5_eq:l_5_1}
		\langle f, Kg\rangle_{L_2} = \int_R^\infty\int_R^\infty
		\left (\int_0^\infty a(t)h_1(x, t)h_2(y, t)dt\right )
		g(y)dy f^*(x)dx =\\
		=\int_0^\infty a(t)\left ( \int_R^\infty h_1(x, t)f^*(x)
		dx\right )\left ( \int_R^\infty h_2(y, t)g(y)dy\right) dt 
		=\\= \int_0^\infty a(t)\langle f, h^t_1\rangle_{L_2} 
		\langle (h^t_2)^*, g\rangle_{L_2} dt,
	\end{multline}
	where the Fubini theorem may be applied since the 
	function under the integral is absolutely integrable by 
	the assumption of the proposition. Recall that 
	$\mathfrak{B}(L_2[R, \infty)) \simeq \mathcal{J}_1^*(L_2[R, \infty))$ 
	with $A\in\mathfrak{B}(L_2[R, \infty)$ acting by $X\mapsto \tr(AX)$. 
	For the norm of $K$, considering it as a functional on 
	$\mathfrak{B}(L_2[R, \infty))$, using definition of trace 
	and expression \eqref{5_eq:l_5_1} we write 
	\[
		\norm{K}_{\mathcal{J}_1} = \sup_{B\in\mathfrak{B}(L_2), 
			\norm{B}=1}\abs{\tr(BK)} =\\= \sup_{B\in\mathfrak{B}(L_2), 
			\norm{B}=1}\abs*{\sum_{i\in \N}\int_0^\infty a(t)\langle f_i, 
			B h_1^t\rangle_{L_2}\langle (h_2^t)^*, f_i\rangle_{L_2}dt},
	\]
	where $\{f_i\}_{i\in\N}$ is an arbitrary orthonormal 
	basis in $L_2[R, \infty)$. We next use the Cauchy-Bunyakovsky-Schwarz 
	inequality to obtain
	\[
		\sum_{i\in \N}\abs{\langle f_i, B h_1^t\rangle_{L_2}
			\langle (h_2^t)^*, f_i\rangle_{L_2}} \le\\ \le 
			\left (\sum_{i\in\N}\abs{\langle f_i, B h_1^t\rangle_{L_2}}^2\right)^{1/2}
			\left (\sum_{i\in\N}\abs{\langle (h_2^t)^*, f_i\rangle_{L_2}}^2\right)^{1/2} 
			\le \norm{h_1^t}_{L_2}\norm{h_2^t}_{L_2},
	\]
	which finishes the proof.
\end{proof}
\begin{proposition}\label{4:trunc_trace}
	Let a kernel $K(x, y)$ of an integral operator 
	$K$ be absolutely continuous with respect to $y$ on 
	$[a, b]$. Assume that $K = K\chi_{[a, b]}$. Consider 
	the operator $\partial_yK$ with the kernel 
	$\frac{\partial K(x, y)}{\partial y}$ and assume 
	that $\partial_y K\in\mathcal{J}_2$. Then we have
	\begin{equation}\label{4:trunc_trace_eq}
		\norm{K}_{\mathcal{J}_1}\le\norm{K} + 
		\frac{(b-a)}{\sqrt{2}}\norm{\partial_y K}_{\mathcal{J}_2}.
	\end{equation}
\end{proposition}
\begin{proof}
	Let $Pf(x) = \frac{1}{b-a}\chi_{[a, b]}(x)\langle 
	\chi_{[a, b]}, f\rangle_{L_2}$ be a one-dimensional
	 projector onto $\chi_{[a, b]}$ and $Q=\chi_{[a, b]}-P$. 
	 Decompose the operator $K = KP + KQ$. The operator $KP$ 
	 is again a one-dimensional projector with the trace norm of at most
	\[
		\norm{KP}_{\mathcal{J}_1}\le \norm{K}\norm{P}_{\mathcal{J}_1} = \norm{K}.
	\]
	
	Introduce the Volterra operator $V$ on $L_2([a, b])$ by the 
	formula $Vf(x) = \int_a^x f(t)dt$. The Volterra operator is a 
	Hilbert-Schmidt operator with the kernel $V(x, y) = 
	\chi_{[a, x]}(y)$ and its Hilbert-Schmidt norm is equal to
	\[
		\norm{V}_{\mathcal{J}_2}^2 = \int_a^bdx\int_a^xdy = \frac{(b-a)^2}{2}.
	\]
	
	We can now express the remaining $KQ$ via integration 
	by parts as follows
	\[
		\int_a^b K(x, y)Qf(y)dy = \int_a^b K(x, y) d(VQf(y)) = 
		K(x, y)VQf(y)\bigg\rvert_a^b - \partial_yKVQf,
	\]
	where 
	\[
		VQf(a) = 0,\qquad VQf(b) = \langle \chi_{[a, b]}, 
		Qf\rangle_{L_2} = 0
	\]
	by the definition of $V$ and $Q$. Inequality 
	$\norm{\partial_yKVQ}_{\mathcal{J}_1}\le 
	\norm{\partial_yK}_{\mathcal{J}_2}\norm{V}_{\mathcal{J}_2}$ 
	finishes the proof. 
\end{proof}
\begin{corrolary}\label{4:trunc_corr}
	There exists a constant $C$ such that for any kernel 
	$K$ of an integral operator on $L_2(\R_+)$ with the 
	following estimate for some constant $A$
	\[
		\abs{K(x, y)} \le \frac{A}{(x+y)^2},\quad 
		\abs{\partial_y K(x, y)}\le \frac{A}{(x+y)^2}, \quad x, y>0
	\]
	we have for any $R>0$
	\begin{equation}
		\norm{\chi_{[R, \infty)}K\chi_{[R, \infty)}}_{\mathcal{J}_1} 
		\le AC\left(\frac{1}{R}+ \frac{1}{\sqrt{R}}\right).
	\end{equation}
\end{corrolary}
\begin{proof}
	By Proposition \ref{4:trunc_trace} we have that
	\begin{multline*}
		\norm{\chi_{[R, \infty)}K\chi_{[R, \infty)}}_{\mathcal{J}_1} 
		\le \sum_{l=0}^\infty \norm{\chi_{[R, \infty)}K
			\chi_{[R + l,R + l+1]}}_{\mathcal{J}_1} \le\\
		\le \sum_{l=0}^\infty(\norm{\chi_{[R, \infty)} K 
			\chi_{[R+l, R+l+1]})}_{\mathcal{J}_2} + 
			\norm{\chi_{[R, \infty)} \partial_yK 
				\chi_{[R+l, R+l+1]})}_{\mathcal{J}_2}),
	\end{multline*}
	where for each summand we have by the assertions
	\begin{multline*}
		\norm{\chi_{[R, \infty)} K \chi_{[R+l, 
				R+l+1]})}_{\mathcal{J}_2}^2 + 
			\norm{\chi_{[R, \infty)} \partial_yK 
				\chi_{[R+l, R+l+1]})}_{\mathcal{J}_2}^2 \le\\
		\le 2A^2\int_R^\infty dx \int_{R+l}^{R+l+1} dy 
		\frac{1}{(x+y)^4} = 3 A^2
		\frac{4R + 2l + 1}{(2R+l)^2(2R+l+1)^2} 
		\le \frac{6A^2}{(2R+l)^2(2R+l+1)}.
	\end{multline*}
	Therefore the following holds for the trace norm
	\[
		\norm{\chi_{[R, \infty)}K\chi_{[R, \infty)}}_{\mathcal{J}_1} 
		\le \frac{\sqrt{3}A}{R} + 2\sqrt{3}A\sum_{l=1}^\infty 
		\frac{1}{(2R+l)^{3/2}}.
	\]
	Lastly, we have that
	\[
		\sum_{l=1}^\infty \frac{1}{(2R+l)^{3/2}} 
		\le \int_{\R_+}\frac{1}{(2R+x)^{3/2}}dx = \frac{1}{\sqrt{2R}},
	\]
	which finishes the proof.
\end{proof}

Let us recall several properties of the Bessel functions. 
Denote $\mathfrak{J}(x) = \sqrt{x}J_\nu(x)$ and 
$\mathfrak{D}(x) = \mathfrak{J}_\nu(x)-\sqrt{2/\pi}\cos(x-\phi_\nu)$, 
where $\phi_\nu = \frac{\pi}{4} + \frac{\pi}{2}\nu$. We have the 
following asymptotics for the Bessel function and its 
derivative as $x$ approaches infinity.
\begin{equation}\label{4:bess_asymptotic}
	\mathfrak{D}(x) = -\sqrt{\frac{2}{\pi}}
	\sin(x-\phi_\nu)\frac{\nu^2-\frac{1}{4}}{x} + O(x^{-2})
\end{equation}
\begin{equation}\label{4:bess_asymptotic_der}
	\mathfrak{D}'(x) = A_\nu\frac{\cos(x-\phi_\nu)}{x} + O(x^{-2}),
\end{equation}
where $A_\nu$ is some constant. These imply uniform 
on $\R_+$ estimates for $\nu> -1$ and some constant $C_\nu$
\begin{equation}\label{4:bessel_uniform}
	\abs{\mathfrak{D}(x)} \le \frac{C_\nu}{\sqrt{x}(1+\sqrt{x})},
\end{equation}
\begin{equation}\label{4:bessel_derivative}
	\abs{\mathfrak{D}'(x) - A_\nu\frac{\cos(x-\phi_\nu)}{x}} 
	\le \frac{C_\nu}{x^{3/2}(1+\sqrt{x})}.
\end{equation}
Recall the following improper integral.
\begin{proposition}[{\cite[Lemma 2.6]{Basor_2003}}]
	We have that
	\begin{equation}\label{4:bessel_formula}
		\int_0^\infty \left( \mathfrak{J}(xt)\mathfrak{J}(yt) 
		- \frac{2}{\pi}\cos(xt-\phi_\nu)\cos(yt-\phi_\nu) 
		\right)dt = -\frac{\sin(2\phi_\nu)}{\pi(x+y)}
	\end{equation}
\end{proposition}
Lastly, we establish certain bounds by $\normB{\cdot}$.
\begin{lemma}\label{4:B_space_prop}
	There exists a constant $C$ such that 
	for any $a\in L_1(\R_+)\cap L_\infty(\R_+)$ 
	satisfying $\normB{a}<\infty$ we have
	
	1. $\norm{a''}_{L_1} \le C(\normH{1}{a} + 
	\normH{2}{a} + \normH{2}{ta(t)})$
	
	2. $\norm{ta'''(t)}_{L_1} \le C(\normH{3}{a} + 
	\normH{1}{a} + \normH{2}{ta(t)} + \normH{3}{t^2a(t)})$
	
	3. $\lim_{t\to+\infty}a(t) = 0$, $\lim_{t\to+\infty}a'(t)=0$, 
	$\lim_{t\to+\infty}ta''(t)=0$.
	
	4. $\abs{a'(0)} \le \normH{1}{a} + \normH{2}{a}$.
	
	In particular, these estimates are bounded by $C\normB{a}$.
\end{lemma}
\begin{proof}
	1. Using the Cauchy-Bunyakovsky-Schwarz inequality 
	write for the $L_1$ norm
	\[
		\norm{a''}_{L_1} = \int_0^1\abs{a''(t)}dt + 
		\int_1^\infty \abs{a''(t)}dt \le \normH{2}{a} + 
		\norm{ta''(t)}_{L_2}.
	\]
	Since $ta''(t) = (at)'' - 2a'$ the second term 
	is estimated by $2\normH{1}{a} + \normH{2}{ta(t)}$.
	
	2. The proof is completely parallel 
	to the previous one.
	
	3. The statement for $a$ follows 
	from $\hat{a}\in L_1(\R)$ and the Riemann-Lebesgue lemma.
	
	Since $a'\in H_1(\R_+)$ we have that $\widehat{a'}$ is 
	absolutely integrable and the statement for $a'$ 
	again follows from the Riemann-Lebesgue lemma.
	
	By the first and second statements $(ta''(t))' = a''(t) + ta'''(t)$ 
	is absolutely integrable, so $ta''(t)$ tends to a finite limit 
	as $t$ approaches infinity. Since $(ta(t))''$ and $a'(t)$ are 
	square integrable, so is $ta''(t)$, which yields that the limit 
	is zero again.
	
	4. Expression of $a'(0)$ via the cosine transform 
	yields the following estimate
	\[
		\abs{b'(0)} \le \frac{1}{\pi}\int_0^\infty 
		\lambda\abs{\hat{b}(\lambda)}d\lambda,
	\]
	where by the Cauchy-Bunyakovsky-Schwarz inequality 
	and the Parseval theorem we get
	\[
		\int_0^1\lambda\abs{\hat{b}(\lambda)}d\lambda 
		\le \normH{1}{b}\quad\quad \int_1^\infty 
		\frac{\lambda^2}{\lambda}\abs{\hat{b}(\lambda)}
		d\lambda \le \normH{2}{b}.
	\]
\end{proof}
We devote the rest of this section to the proof 
of Lemma \ref{2:bess_diff_est}.
\begin{proof}[Proof of Lemma \ref{2:bess_diff_est}]
	Since $B_1 - W_1 = 0$, it is enough to proof the statement 
	for $b\in L_1(\R_+)\cap L_\infty(\R_+)$, satisfying $\normB{b}<\infty$. 
	Recall the formula for the kernel of the difference 
	$\mathcal{R}_b(x, y) = B_b(x, y) - W_b(x, y)$:
	\[
		\mathcal{R}_b(x, y) = \int_0^\infty \left( 
		\mathfrak{J}(xt)\mathfrak{J}(yt) - 
		\frac{1}{\pi}\cos( (x-y)t) \right)b(t)dt.
	\]
	Let us outline the plan of the proof. One can observe 
	from asymptotics \eqref{4:bess_asymptotic} and 
	\eqref{4:bess_asymptotic_der} that the integral above 
	contains difference of two asymptotically similar functions. 
	Therefore, it is reasonable to substitute $\mathfrak{J}(x) = 
	\sqrt{x}J_\nu(x) = \mathfrak{D}(x) + \sqrt{2/\pi}\cos(x-\phi_\nu)$ 
	into the expression for $\mathcal{R}_b(x, y)$. This substituion 
	and several integrations by parts represent the kernel as a 
	sum of different kernels, to which Corollary \ref{4:trunc_corr} 
	and Proposition \ref{4:est_separable} may be applied.
	
	To be precise, let us first introduce the notation. 
	We will do a sequence of decompositions of our 
	kernel, which we denote as follows
	\begin{equation}\label{4_eq:dec_1}
		\mathcal{R}_b(x, y) = 
		\mathcal{R}_1(x, y) - \mathcal{R}_2(x, y),
	\end{equation}
	\begin{equation}\label{4_eq:dec_2}
		\mathcal{R}_2(x, y) = 
		S(x, y) + T(x, y) + T(y, x),
	\end{equation}
	\begin{equation}\label{4_eq:dec_3}
		T(x, y) = T_0(x, y) +
		 T_1(x, y) + Z(x, y).
	\end{equation}
	Explicit formulae in the notation above will be given 
	below (see formulae \eqref{4_eq:dec_1_done}, 
	\eqref{4_eq:dec_2_done}, \eqref{4_eq:dec_3_done}). 
	We prove the estimate for trace norm of at most 
	constant times $\normB{b}/\sqrt{R}$ for each of 
	the introduced kernels separately. In particular, 
	the estimate of trace norm for $\chi_{[R, \infty)}
	\mathcal{R}_1\chi_{[R, \infty)}$ will follow from 
	Corollary \ref{4:trunc_corr}. Estimates on trace 
	norms of operators with the following kernels 
	\begin{align*}
		&\chi_{[R, \infty)^2}(x, y)S(x, y),\quad
		\chi_{[R, \infty)^2}(x, y)T_0(x, y),\\
		&\chi_{[R, \infty)^2}(x, y)T_1(x, y),\quad
		\chi_{[R, \infty)^2}(x, y)(Z(x, y) + Z(y, x))
	\end{align*}
	will follow from Proposition \ref{4:est_separable}.
	
	\textbf{Sequence of decompositions}
	
	\textbf{1. Decomposition \eqref{4_eq:dec_1}}
	
	First substitute the following formula into the kernel 
	$\mathcal{R}_b(x, y)$
	\[
		\frac{1}{\pi}\cos( (x-y)t) = \frac{2}{\pi}
		\cos(xt-\phi_\nu)\cos(yt-\phi_\nu) - 
		\frac{1}{\pi}\cos( (x+y)t-2\phi_\nu).
	\]
	By the third statement of Lemma \ref{4:B_space_prop} 
	the following integral may be integrated by parts 
	two times and, hence, expressed as follows
	\begin{multline*}
		\frac{1}{\pi}\int_0^\infty \cos( (x+y)t - 2\phi_\nu)b(t)dt =\\
		= \frac{\sin(2\phi_\nu)b(0)}{\pi(x+y)}
		- \frac{b'(0)\cos(2\phi_\nu)}{\pi(x+y)^2} - 
		\frac{1}{\pi(x+y)^2}\int_0^\infty \cos( (x+y)t - 
		2\phi_\nu)b''(t)dt.
	\end{multline*}
	Next we substitute expression \eqref{4:bessel_formula} 
	for $ \sin(2\phi_\nu)/(\pi(x+y))$ into the identity above.
	 These calculations prove decomposition \eqref{4_eq:dec_1} 
	 for the following $\mathcal{R}_1, \mathcal{R}_2$
	\begin{equation}\label{4_eq:dec_1_done}
		\begin{aligned}  
			&\mathcal{R}_1(x, y) = \frac{1}{\pi(x+y)^2}
			\left(\cos(2\phi_\nu)b'(0) + \int_0^\infty
			\cos( (x+y)t-2\phi_\nu)b''(t)dt\right),\\
			&\mathcal{R}_2(x, y) = \int_0^\infty 
			\left( \mathfrak{J}(xt)\mathfrak{J}(yt) - 
			\frac{2}{\pi}\cos(xt-\phi_\nu)\cos(yt-\phi_\nu) 
			\right)b_0(t)dt,
		\end{aligned}
	\end{equation}
	where $b_0(x) = b(x) - b(0)$.
	
	\textbf{2. Decomposition \eqref{4_eq:dec_2}}
	
	It may be directly verified that
	\eqref{4_eq:dec_2} holds if we take $S$, $T$ to be
	\begin{equation}\label{4_eq:dec_2_done}
		\begin{aligned}
			&S(x, y) = \int_0^\infty \mathfrak{D}(xt)
			\mathfrak{D}(yt)b_0(t)dt,\\
			&T(x, y) = \int_0^\infty \mathfrak{D}(xt)
			\sqrt{\frac{2}{\pi}}\cos(yt-\phi_\nu)b_0(t)dt.
		\end{aligned}
	\end{equation}
	
	\textbf{3. Decomposition \eqref{4_eq:dec_3}}
	
	Integrate by parts the expression for $T(x, y)$
	\begin{multline*}
		\int_0^\infty \mathfrak{D}(xt)\sqrt{\frac{2}{\pi}}
		\frac{1}{y}b_0(t)d(\sin(yt-\phi_\nu)) = \mathfrak{D}(xt)
		\sqrt{\frac{2}{\pi}}\frac{1}{y}b_0(t)\sin(yt-\phi_\nu)
		\bigg\rvert_0^{\infty} - \\ - \int_0^\infty \mathfrak{D}(xt)
		\sqrt{\frac{2}{\pi}}\frac{1}{y}\sin(yt-\phi_\nu)b'(t)dt -
		 \int_0^\infty x\mathfrak{D}'(xt)\sqrt{\frac{2}{\pi}}
		 \frac{1}{y}\sin(yt-\phi_\nu)b_0(t)dt,
	\end{multline*}
	where the first term is zero by Lemma \ref{4:B_space_prop} 
	and estimate \eqref{4:bessel_uniform}. We next take the 
	following kernels the decomposition \eqref{4_eq:dec_3}
	\begin{equation}\label{4_eq:dec_3_done}
		\begin{aligned}
			&T_0(x, y) = - \int_0^\infty \mathfrak{D}(xt)
			\sqrt{\frac{2}{\pi}}\frac{1}{y}\sin(yt-\phi_\nu)b'(t)dt,\\
			&T_1(x, y) = - \int_0^\infty x\left(\mathfrak{D}'(xt) 
			- A_\nu\frac{\cos(xt-\phi_\nu)}{xt}\right)
			\sqrt{\frac{2}{\pi}}\frac{1}{y}\sin(yt-\phi_\nu)b_0(t)dt,\\
			&Z(x, y) = - \int_0^\infty x A_\nu\frac{\cos(xt-\phi_\nu)}{xt}
			b_0(t)\sqrt{\frac{2}{\pi}}\frac{1}{y}\sin(yt-\phi_\nu)dt.
		\end{aligned}
	\end{equation}
	
	\textbf{Trace norm estimates}
	
	Before diving into calculations we give several inequalities for $b$. 
	Firstly, the Cauchy-Bunyakovsky-Schwarz inequality implies the following
	\begin{equation}\label{4:volt_est}
		\abs{b_0(t)} = \abs*{\int_0^tb'(x)dx} \le \sqrt{t}\normH{1}{b}.
	\end{equation}
	The same argument for derivative gives
	\begin{equation}\label{4:volt_est_der}
		\abs{b'(t)} \le \sqrt{t}(\abs{b'(0)} + \normH{2}{b}).
	\end{equation}
	Further, using inequality \eqref{4:volt_est_der} we have
	\begin{equation}\label{4:volt_est_str}
		\abs{b_0(t)} = \abs*{\int_0^tb'(x)dx} \le 
		(\abs{b'(0)} + \normH{2}{b})t^{3/2}.
	\end{equation}
	Lastly, observe the identity
	\[
		\frac{d}{dt}\left (\frac{b_0(t)}{t}\right ) = 
		\frac{1}{t^2}\left(\int_0^tb'(x)dx - b_0(t) + 
		\int_0^txb''(x)dx\right) = \frac{1}{t^2}\int_0^txb''(x)dx.
	\]
	The Cauchy-Bunyakovsky-Schwarz inequality implies 
	$\abs*{\int_0^txb''(x)du} \le t^{3/2}\normH{2}{b}$. 
	This inequality with the expression above yield
	\begin{equation}\label{4:frac_est}
		\abs*{\frac{d}{dt}\left (\frac{b_0(t)}{t}\right )}
		\le \frac{\normH{2}{b}}{\sqrt{t}}.
	\end{equation}
	
	\textbf{1. Estimate for $\mathcal{R}_1$}
	
	We immediately have
	\[
		\abs{\mathcal{R}_1(x, y)} \le 
		\frac{1}{(x+y)^2}(\abs{b'(0)} + \norm{b''}_{L_1}).
	\]
	By Lemma \ref{4:B_space_prop} $b''$ is absolutely integrable, 
	so we have the following expression for the derivative
	\begin{multline*}
		\partial_y\mathcal{R}_1(x, y) = -\frac{2}{\pi(x+y)^3}
		\left(\cos(2\phi_\nu)b'(0) + \int_0^\infty\cos( (x+y)t-2\phi_\nu)
		b''(t)dt\right) -\\
		- \frac{1}{\pi(x+y)^2}\int_0^\infty\sin( (x+y)t-
		2\phi_\nu)tb''(t)dt.
	\end{multline*}
	We next integrate the second term by parts
	\begin{multline*}
		\frac{1}{(x+y)}\int_0^\infty d(\cos( 
		(x+y)t-2\phi_\nu))tb''(t)dt = \cos((x+y)t - 2\phi_\nu)
		tb''(t)\bigg\rvert_0^{\infty} -\\
		-\frac{1}{(x+y)}\int_0^\infty \cos( (x+y)t-2\phi_\nu)
		(tb'''(t) + b''(t))dt,
	\end{multline*}
	where the limit in infinity is zero by Lemma \ref{4:B_space_prop}. 
	These calculations imply the following estimate for the derivative for $x, y\ge 1$
	\[
		\abs{\partial_y\mathcal{R}_1(x, y)} \le \frac{2}{(x+y)^2}
		(\abs{b'(0)} + \norm{b''}_{L_1} + \norm{tb'''(t)}_{L_1}).
	\]
	Finally, applying Corrollary \ref{4:trunc_corr} and Lemma 
	\ref{4:B_space_prop} we get the desired estimate for some 
	constant $C$ and $R\ge 1$
	\[
		\norm{\chi_{[R, \infty)}\mathcal{R}_1\chi_{[R, \infty)}}_{\mathcal{J}_1} 
		\le \frac{C\normB{b}}{\sqrt{R}}.
	\]
	
	\textbf{2. Estimate for $S$}
	
	Using estimate \eqref{4:bessel_uniform} we get
	\[
		\int_R^\infty \abs{\mathfrak{D}(xt)}^2 dx \le C_\nu^2\int_R^\infty 
		\frac{1}{xt(1+\sqrt{xt})^2}dx = \frac{2C_\nu^2}{t}\left(\ln
		\left(1+\frac{1}{\sqrt{tR}}\right)-\frac{1}{1+\sqrt{tR}}\right).
	\]
	Denote $G(x) = \ln(1+1/\sqrt{x})-1/(1+\sqrt{x})$. Using 
	Proposition \ref{4:est_separable} we have
	\[
		\norm{\chi_{[R, \infty)}S\chi_{[R, 
				\infty)}}_{\mathcal{J}_1} \le 2C_\nu^2
			\int_0^\infty \frac{\abs{b_0(t)}}{t}G(tR)dt.
	\]
	We next substitute inequality \eqref{4:volt_est} to obtain
	\[
		\norm{\chi_{[R, \infty)}S\chi_{[R, \infty)}}_{\mathcal{J}_1}
		 \le 2C_\nu^2\normH{1}{b}\int_0^\infty \frac{G(tR)}{\sqrt{t}}dt
		  = \frac{4C_\nu^2}{\sqrt{R}}\normH{1}{b}\le 
		  \frac{4C_\nu^2}{\sqrt{R}}\normB{b}.
	\]
	
	\textbf{3. Estimate for $T_0$}
	
	Again Proposition \ref{4:est_separable} together with estimate 
	\eqref{4:bessel_uniform} give
	\[
		\norm{\chi_{[R, \infty)}T_0\chi_{[R, \infty)}}_{\mathcal{J}_1} 
		\le \frac{C_\nu}{\sqrt{R}}\int_0^\infty 
		\frac{\abs{b'(t)}}{\sqrt{t}}\sqrt{G(tR)}dt.
	\]
	For the integral on $[0, 1]$ using inequality 
	\eqref{4:volt_est_der} write for some constant $C$
	\[
		\int_0^1\frac{\abs{b'(t)}}{\sqrt{t}}G^{1/2}(tR)dt \le 
		(\abs{b'(0)} + \normH{2}{b})\int_0^1\sqrt{G(tR)}dt \le 
		C(\abs{b'(0)} + \normH{2}{b}).
	\]
	And for the integral on $[1, \infty)$ use the 
	Cauchy-Bunyakovsky-Schwarz inequality
	\[
		\int_1^\infty\frac{\abs{b'(t)}}{\sqrt{t}}\sqrt{G(tR)}dt 
		\le \normH{1}{b}\sqrt{\int_1^\infty \frac{G(xR)}{x}dx} 
		\le \sqrt{2}\normH{1}{b}.
	\]
	These calculations yield the following estimate by 
	Lemma \ref{4:B_space_prop}
	\[
		\norm{\chi_{[R, \infty)}T_0\chi_{[R, \infty)}}_{\mathcal{J}_1} 
		\le \frac{2CC_\nu}{\sqrt{R}}(\abs{b'(0)} + \normH{1}{b} + 
		\normH{2}{b}) \le \frac{4CC_\nu}{\sqrt{R}}\normB{b}.
	\]
	
	\textbf{4. Estimate for $T_1$}
	
	Using estimate \eqref{4:bessel_derivative} we get
	\[
		\int_R^\infty x^2\left (\mathfrak{D}'(xt) - 
		A_\nu\frac{\cos(xt - \phi_\nu)}{xt}\right)^2
		dx \le \frac{C_\nu^2}{t^3}\int_{R}^\infty 
		\frac{1}{x(1+\sqrt{xt})^2}dx = 
		\frac{2C_\nu^2}{t^3}G(tR).
	\]
	Together with Proposition \ref{4:est_separable} this implies
	\[
		\norm{\chi_{[R, \infty)}T_1\chi_{[R, \infty)}}_{\mathcal{J}_1} 
		\le \frac{\sqrt{2}C_\nu}{\sqrt{R}}\int_0^\infty 
		\frac{\abs{b_0(t)}}{t^{3/2}}\sqrt{G(tR)}dt.
	\]
	For the integral on $[0, 1]$ use inequality \eqref{4:volt_est_str} 
	to obtain the following for some constant $C$
	\[
		\int_0^1 \frac{\abs{b_0(t)}}{t^{3/2}}\sqrt{G(tR)}dt 
		\le (\abs{b'(0)} + \normH{2}{b})\int_0^1\sqrt{G(tR)}dt 
		\le C(\abs{b'(0)} + \normH{2}{b}).
	\]
	For the integral on $[1, \infty)$ using inequality \eqref{4:volt_est}
	 we have for some constant $\tilde{C}$
	\[
		\int_1^\infty \frac{\abs{b_0(t)}}{t^{3/2}}\sqrt{G(tR)}dt 
		\le \normH{1}{b}\int_1^\infty\frac{\sqrt{G(tR)}}{t}dt 
		\le \tilde{C}\normH{1}{b}.
	\]
	Therefore, we conclude the following
	\[
		\norm{\chi_{[R, \infty)}T_1\chi_{[R, \infty)}}_{\mathcal{J}_1} 
		\le \frac{\sqrt{2}(C+\tilde{C})C_\nu}{\sqrt{R}}(\normH{1}{b} + 
		\normH{2}{b} + \abs{b'(0)}) \le 
		\frac{2\sqrt{2}(C+\tilde{C})C_\nu}{\sqrt{R}}\normB{b}.
	\]
	
	\textbf{5. Estimate for $Z(x, y) + Z(y, x)$}
	
	Denote the respective operator by $\tilde{Z}$. 
	Integrate its kernel by parts as follows
	\begin{multline*}
		\tilde{Z}(x, y) = Z(x, y) + Z(y, x) = A_\nu\sqrt{\frac{2}{\pi}}\int_0^\infty
		 \frac{b_0(t)}{xyt}\frac{d}{dt}(\sin(xt-\phi_\nu)\sin(yt-\phi_\nu))dt =\\
		= -A_\nu\sqrt{\frac{2}{\pi}}\frac{b'(0)}{xy}\sin^2(\phi_\nu) -
		 A_\nu\sqrt{\frac{2}{\pi}}\frac{1}{xy}\int_0^\infty\frac{d}{dt}
		 \left (\frac{b_0(t)}{t}\right)\sin(xt - \phi_\nu)
		 \sin(yt-\phi_\nu)dt.
	\end{multline*}
	The first term is a kernel of the form $h_1(x)\langle h_2(y), -\rangle_{L_2}$. 
	Trace norm of the respective operator is bounded by $\norm{h_1}_{L_2}\norm{h_2}_{L_2}$. 
	Therefore trace norm of the operator on $L_2[R, \infty)$, corresponding to the 
	first term, is estimated by $\abs{A_\nu b'(0)}/R \le \abs{A_\nu}\normB{b}/R$ 
	by the last statement of Lemma \ref{4:B_space_prop}. We employ Proposition 
	\ref{4:est_separable} for the operator with the kernel given by the
	 second term. This results in the following estimate
	\[
		\norm{\chi_{[R, \infty)}\tilde{Z}\chi_{[R, \infty)}}_{\mathcal{J}_1} 
		\le \frac{\abs{A_\nu}\normB{b}}{R} + \frac{\abs{A_\nu}}{R}\int_0^\infty
		\abs*{\frac{d}{dt}\left (\frac{b_0(t)}{t}\right)}dt.
	\]
	The integral on $[0, 1]$ is estimated via inequality \eqref{4:frac_est}
	\[
		\int_0^1\abs*{\frac{d}{dt}\left (\frac{b_0(t)}{t}\right)}dt 
		\le \normH{2}{b}.
	\]
	The integral on $[1, \infty)$ is estimated using inequality 
	\eqref{4:volt_est} and the Cauchy-Bunyakovsky-Schwarz inequality
	\[
		\int_1^\infty\abs*{\frac{d}{dt}\left (\frac{b_0(t)}{t}\right)}dt 
		\le \int_1^\infty\left (\frac{\abs{b'(t)}}{t} + 
		\frac{\normH{1}{b}}{t^{3/2}}\right )dt \le 3\normH{1}{b}.
	\]
	Therefore, we conclude the following estimate
	\[
		\norm{\chi_{[R, \infty)}\tilde{Z}\chi_{[R, \infty)}}_{\mathcal{J}_1} 
		\le \frac{5\abs{A_\nu}\normB{b}}{R}.
	\]
	This finishes the proof of Lemma \ref{2:bess_diff_est}.
\end{proof}

\section{Proof of Theorems \ref{1:BO_bessel} and \ref{1:KS_estimate}}
Recall the Jacobi-Dodgson identity.
\begin{proposition}[ 
	{\cite[Proposition~6.2.9]{simon2005orthogonal}}]\label{5:JD_form}
	Let $A$ be a determinant class invertible operator on a 
	separable Hilbert space $\mathcal{H}$. Let $P$ be an operator 
	of orthogonal projection and $Q=I-P$. Then the following relation holds
	\begin{equation}
		\det(PAP) = \det( A)\det(QA^{-1}Q).
	\end{equation}
\end{proposition}
We use the following variation of this identity.
\begin{corrolary}\label{5:JD_ext}
	Let $P_1, P_2$ be commuting orthogonal projectors in 
	a separable Hilbert space $\mathcal{H}$. 
	Let $Q_i=I-P_i$, $i=1, 2$. The following relation holds
	\begin{equation}\label{5:JD_ext_eq}
		\frac{\det(P_1AP_1)}{\det(Q_1A^{-1}Q_1)}=
		\frac{\det(P_2AP_2)}{\det(Q_2A^{-1}Q_2)}
	\end{equation}
	if $A$ is invertible, $A-I$ is Hilbert-Schmidt and 
	all present determinants are well defined.
\end{corrolary}
\begin{remark}
	The statement does not require $A-I$ to be trace class.
\end{remark}
\begin{proof}
	Let $A = I + K$, $K$ is Hilbert-Schmidt. Choose joint orthogonal basis 
	$\{e_i\}_{i\in\N}$ of eigenvectors of $P_1, P_2$. Let $R_n$ 
	be an operator of orthogonal projection on 
	$\text{span}(\{e_i\}_{i\in 1..n})$. Then $R_n\mathcal{H}$ is 
	a separable Hilbert space. Since $P_i, Q_i$ commute with $R_n$, 
	they remain orthogonal projectors in $R_n\mathcal{H}$. 
	
	Since $K$ is compact, $R_nKR_n \to K$ in operator norm. 
	This implies that for all large enough $n$ the operator 
	$I + R_n K R_n$ is invertible. The operator $R_nKR_n$ is 
	finite-dimensional, so applying Proposition \ref{5:JD_form} we have
	\begin{equation}\label{5:JD_eq_finite}
		\frac{\det(R_nP_1AP_1R_n)}{\det(R_nQ_1A^{-1}Q_1R_n)}=
		\frac{\det(R_nP_2AP_2R_n)}{\det(R_nQ_2A^{-1}Q_2R_n)} = \det(R_nAR_n).
	\end{equation}
 We next show that $(R_nAR_n)^{-1}\to R_nA^{-1}R_n\to 0$ as $n\to \infty$ in trace norm. We have for the difference
 \[
 (R_n+R_nKR_n)^{-1}R_n(I+K)^{-1}R_n = (R_nKR_n)^2(R_n+R_nKR_n)^{-1}-R_nK^2(I+K)^{-1}R_n,
 \]
 where $(R_n + R_nKR_n)^{-1}\to (I+K)^{-1}$, $R_n\to I$ strongly and $(R_nKR_n)^2\to K$, $R_nK^2\to K^2$ in trace norm since $K$ is Hilbert-Schmidt. Therefore, we have
 \[
\lim_{n\to\infty}\det(Q_i(R_nAR_n)^{-1}Q_i)=\lim_{n\to\infty}\det(R_nQ_iA^{-1}Q_iR_n)=\det(Q_iA^{-1}Q_i).
 \]
 The equality \eqref{5:JD_ext_eq} follows by taking limits in relation \eqref{5:JD_eq_finite} as $n\to\infty$, since $\det(R_nP_iAP_iR_n)\to \det(P_iAP_i)$.
\end{proof}

Let us extend definition \eqref{1:bess_def} to all 
functions from $L_\infty (\R_+)$. Introduce the Hankel transform 
$H_\nu$ on $L_2(\R_+)$ by the following formula
\[
	H_\nu f(\lambda) = \int_0^\infty \sqrt{\lambda x}J_\nu (\lambda x) f(x)dx.
\]
As the Fourier transform, we firstly define it on $L_1(\R_+)\cap L_2(\R_+)$, 
prove that it is an isometry and extend the operator by continuity. 
In particular, we have $H_\nu^* = H_\nu^{-1} = H_\nu$. 
Now for $b\in L_\infty(\R_+)$ we define the Bessel operator to be
\begin{equation}\label{5:bess_def_inf}
	B_b = H_\nu b H_\nu.
\end{equation}
This definition coincides with definition \eqref{1:bess_def} for 
$b\in L_1(\R_+)\cap L_\infty(\R_+)$, but now it is clear that 
$\mathbf{B}:f\to B_f$, $L_\infty(\R_+)\to \mathfrak{B}(L_2(\R_+))$ 
is a Banach algebra homomorphism and in particular that $B_{e^b}^{-1} = B_{e^{-b}}$.

Next we want to show that for $b\in \mathcal{B}$ the operator 
$B_b - W_b$ is Hilbert-Schmidt. As shown by Lemma \ref{2:bess_diff_est}, 
$\chi_{[1, \infty)}(B_b - W_b)\chi_{[1, \infty)}$ is trace class. 
The operator $\chi_{[0, 1]}(B_b - W_b)\chi_{[0, 1]}$ is trace class 
by Theorem \ref{3_0:mercer_th}. The following statement may be directly verified.
\begin{lemma} [{\cite[Lemmata 3.1, 3.2]{Basor_2003}}]
	For $b\in H_1(\R_+)\cap L_1(\R_+)$ we have 
	that $\chi_{[0, 1]}B_b$ and $\chi_{[0, 1]}W_b$ are Hilbert-Schmidt.
\end{lemma}
This concludes that $B_b - W_b$ is Hilbert-Schmidt.
\begin{proof}[Proof of Lemma \ref{2:det_relation}]
	For arbitrary $R_1, R_2 >0$ we put $P_1 = \chi_{[0, R_1]}$, $P_2 = \chi_{[0, R_2]}$, 
	$A = W_{e^{-b_+}}B_{e^b}W_{e^{-b_-}}$. By the first statement of 
	Proposition \ref{3:WH_properties} and from the extended definition of the 
	Bessel operator we have that $A$ is invertible and the inverse is 
	$W_{e^{b_-}}B_{e^{-b}}W_{e^{b_+}}$. We also have that 
	$Q_1 = \chi_{[R_1, \infty)}$, $Q_2 = \chi_{[R_2, \infty)}$. 
	Using second and third statements of Proposition 
	\ref{3:WH_properties} write the following 
	\begin{equation}\label{5:denominator_expr}
		Q_i A^{-1} Q_i = \chi_{[R_i, \infty)} + 
		\chi_{[R_i, \infty)}W_{e^{b_-}}\chi_{[R_i, \infty)}\mathcal{R}_{e^{-b}}
		\chi_{[R_i, \infty)}W_{e^{b_+}}\chi_{[R_i, \infty)},
	\end{equation}
	where $\mathcal{R}_b = B_b - W_b$. Hence by assumptions and due 
	to the shown above fact that $\mathcal{R}_{e^{-b}}$ is compact 
	we can apply Corollary \ref{5:JD_ext} to obtain
	\begin{equation}\label{5:det_class_expr}
		\frac{\det(\chi_{[0, R_1]}W_{e^{-b_+}}B_{e^{b}}W_{e^{-b_-}}
			\chi_{[0, R_1]})}{\det(\chi_{[R_1, \infty)}W_{e^{b_-}}B_{e^{-b}}
			W_{e^{b_+}}\chi_{[R_1, \infty)})} =
		\frac{\det(\chi_{[0, R_2]}W_{e^{-b_+}}B_{e^{b}}W_{e^{-b_-}}
			\chi_{[0, R_2]})}{\det(\chi_{[R_2, \infty)}W_{e^{b_-}}
			B_{e^{-b}}W_{e^{b_+}}\chi_{[R_2, \infty)})}= Z(b),
	\end{equation}
	where all of determinants are well defined by assertions.
\end{proof}

Recall the asymptotic result of Basor and Ehrhardt.
\begin{theorem} [
	{\cite[Theorem~1.1]{Basor_2003}}]\label{5:bessel_asymp}
	Suppose the function $b\in L_1(\R_+)\cap L_\infty(\R_+)$ 
	satisfies the following conditions
	\begin{itemize}
		\item it is continuous and piecewise 
		$C^2$ on $[0, \infty)$, and $\lim_{t\to\infty}b(t)=0$.
		\item $(1+t)^{-1/2}b'(t)\in L_1(\R_+)$, $b''(t)\in L_1(\R_+)$.
	\end{itemize}
	Then the following asymptotic formula holds as $R\to\infty$
	\begin{equation}
		\det(\chi_{[0, R]}B_{e^{b}}\chi_{[0, R]}) = 
		\exp(Rc_1^{\mathcal{B}}(b) + c_2^{\mathcal{B}}(b) + 
		c_3^{\mathcal{B}}(b))Q_R^{\mathcal{B}}(b), 
		\quad Q_R^{\mathcal{B}}(b)\to 1,
	\end{equation}
	where $c_i^\mathcal{B}(b)$ are given as in Theorem \ref{1:BO_bessel}.
\end{theorem}
Clearly assumptions of Theorem \ref{1:BO_bessel} are more 
restrictive by Lemma \ref{4:B_space_prop}. The following 
lemma establishes that $b\in \mathcal{B}$ is sufficient 
for $e^{-b}$ to satisfy conditions of Lemma \ref{2:bess_diff_est}.

\begin{lemma}\label{5:B_space_prop_pl}
	Let $b\in\mathcal{B}$. Then we have that $e^b - 1\in L_1\cap L_\infty$. 
	There exists a constant $C$ such that for any $b\in\mathcal{B}$ we have
	\[
	\normB{e^b}\le Ce^{\norm{b}_{L_\infty}}(1 + \norm{xb'(x)}_{L_\infty}^2 + 
	\normB{b}^2)\normB{b}.
	\]
\end{lemma}
\begin{proof}
	The first statement follows directly from $L_\infty$ being a Banach 
	algebra with pointwise multiplication and
	\[
	\norm{e^b - 1}_{L_1} \le \norm{b}_{L_1}\norm*{\frac{e^b-1}{b}}_{L_\infty}.
	\]
	
	For the estimate we immediately have
	\[
	\normH{1}{e^b} \le e^{\norm{b}_{L_\infty}}\normH{1}{b}.
	\]
	The third derivative of $e^b$ is
	\[
	(e^b)''' = e^b(b''' + (b')^3 + 3b'b'').
	\]
	Observe that 
	$$\norm{b'}_{L_\infty} \le \norm{\lambda 
		\hat{b}(\lambda)}_{L_1} \le \normH{1}{b} + \normH{2}{b}$$
	This yields the following estimate
	\[
	\normH{3}{e^b} \le e^{\norm{b}_{L_\infty}}
	(\normH{3}{b} + \normB{b}^2\normH{1}{b} + 3\normB{b}\normH{2}{b}).
	\]
	
	We next write the second derivative of $xe^{b(x)}$
	\[
	(xe^{b(x)})'' = e^{b(x)}((xb(x))'' + x(b'(x))^2).
	\]
	Therefore the following estimate holds
	\[
	\normH{2}{xe^{b(x)}} \le e^{\norm{b}_{L_\infty}}
	(\normH{2}{xb(x)} + \norm{xb'(x)}_{L_\infty}\normH{1}{b}).
	\]
	
	The third derivative of $x^2e^{b(x)}$ is
	\[
	(x^2e^{b(x)})''' = e^{b(x)}(6x(b'(x))^2 +3x^2b'(x)b''(x) 
	+ (x^2b(x))''' + x^2(b'(x))^3).
	\]
	This implies that
	\[
	\normH{3}{x^2e^{b(x)}}\le e^{\norm{b}_{L_\infty}}
	(3\norm{xb'(x)}_{L_\infty}\normH{2}{xb(x)} + \normH{3}{x^2b(x)} + 
	\norm{xb'(x)}_{L_\infty}^2\normH{1}{b}).
	\]
\end{proof}
\begin{proof}[Proof of Theorem \ref{1:BO_bessel}]
	Recall the following inequality
	\[
		\abs{\det(I+K) - 1} \le \norm{K}_{\mathcal{J}_1}
		\exp(\norm{K}_{\mathcal{J}_1}).
	\]
	Also for any trace class operator $A$ and bounded $B$ we have 
	$\norm{AB}_{\mathcal{J}_1} \le \norm{A}_{\mathcal{J}_1}\norm{B}$. 
	By the definition of Wiener-Hopf operators we have $\norm{W_b} = 
	\norm{b}_{L_\infty}$. These facts and Lemmata \ref{2:bess_diff_est}, 
	\ref{5:B_space_prop_pl} with expression \eqref{5:denominator_expr} 
	prove the estimate \eqref{1:bessel_speed}. 
	
	For the derivation of $Z(b)$ observe that the denominator in 
	expression \eqref{2:det_relation_eq} approaches one by the argument above. 
	From Lemma \ref{2:bessel_div} and Theorem \ref{5:bessel_asymp} we have 
	for the numerator in \eqref{2:det_relation_eq}
	\[
		\det(\chi_{[0, R]}W_{e^{-b_+}}B_{e^{b}}W_{e^{-b_-}}\chi_{[0, R]}) 
		\to \exp(c_2^{\mathcal{B}}(b) + c_3^{\mathcal{B}}(b)), 
		\quad \text{as }R\to \infty.
	\]
	It is now clear that $Z(b) = \exp(c_2^{\mathcal{B}}(b) + 
	c_3^{\mathcal{B}}(b))$. This finishes the proof.
\end{proof}

We now proceed to the proof of Theorem \ref{1:KS_estimate}. 
Firstly we establish an estimate for the speed of convergence 
of the expectation of $S_f^R$. The argument is based on proof 
of Proposition 5.1 in \cite{Basor_2003}.

\begin{lemma}\label{5:exp_conv}
	There exists a constant $C$ such that 
	for any $b\in\mathcal{B}$ we have
	\[
	\abs*{\mathbb{E}_{J_\nu}S_f^R - (\hat{b}(0) - 
		\frac{\nu}{2}b(0))} \le \frac{C\normB{b}}{\sqrt{R}}.
	\]
\end{lemma}
\begin{proof}
	The expression for expectation of additive functional is
	\[
	\mathbb{E}_{J_\nu}S_f^R = \int_{\R_+}f(x/R)
	B_{\chi_{[0, 1]}}(x, x)dx.
	\]
	Recall that the Bessel kernel $B_{\chi_{[0, 1]}}(x, y)$ 
	is defined on the diagonal by the following formula
	\[
	B_{\chi_{[0, 1]}}(x, x) = \frac{x}{2}(J_\nu^2(x) - 
	J_{\nu+1}(x)J_{\nu-1}(x)).
	\]
	The asymptotic of the diagonal as $x\to\infty$ is
	\[
	B_{\chi_{[0, 1]}}(x, x) = \frac{1}{\pi} + 
	\frac{\sin(2(x-\phi_\nu))}{x} + O(x^{-2}).
	\]
	This implies that the following function is well defined
	\[
	F(\xi) = -\int_\xi^\infty\left (\frac{t}{2}(J_\nu^2(t) - 
	J_{\nu+1}(t)J_{\nu-1}(t)) - \frac{1}{\pi}\right)dt.
	\]
	It is clear that $F(\xi) = O(\xi^{-1})$ as $\xi\to\infty$. 
	We therefore have $\abs{F(\xi)} \le C(1+\xi)^{-1}$ for 
	some constant $C$. We use $F$ for integration by parts
	\begin{multline}\label{5_eq:7_7}
		\int_{\R_+}f(x/R)B_{\chi_{[0, 1]}}(x, x)dx - \frac{1}{\pi}\int_{\R_+}f(x/R)dx 
		= f(x/R)F(x)\bigg\rvert_0^{\infty} - \\-\frac{1}{R}\int_{\R_+}f'(x/R)F(x)dx,
	\end{multline}
	where by Lemma \ref{4:B_space_prop} we have $f(\infty)F(\infty)=0$. 
	For the value in zero we write using $J_{\nu-1}(t) + J_{\nu+1}(t) 
	= \frac{2\nu}{t}J_\nu(t)$
	\begin{multline*}
	F(0) = \int_0^\infty \left (\frac{t}{2}(J_\nu^2(t) - 
	J_{\nu+1}(t)J_{\nu-1}(t)) - \frac{1}{\pi}\right)dt =\\= 
	\int_0^\infty\left(\frac{t}{2}(J_\nu^2(t) + J_{\nu+1}^2(t)) 
	- \frac{1}{\pi}\right)dt - \nu\int_0^\infty J_{\nu+1}(t)J_\nu(t)dt.
	\end{multline*}
	The second integral is equal to $1/2$ (see \cite[Sect.~6.512-3]{grad1994}). 
	For the first term use the following integral
	\begin{multline*}
	\int_0^T xJ_\nu^2(x)dx = \frac{T^2}{2}(J_\nu^2(T) - 
	J_{\nu+1}(T)J_{\nu-1}(T))=\\= \frac{T}{\pi}+\sin(2(T-\phi_\nu))+o(1) 
	\quad \text{as }T\to\infty.
	\end{multline*}
	Recall that $\phi_{\nu+1} = \phi_{\nu}+\frac{\pi}{2}$. 
	Therefore the term is zero. We conclude that $F(0)b(0) = 
	\frac{\nu}{2}b(0)$.
	
	What is left to do is estimate the second term in \eqref{5_eq:7_7}. 
	Using $F(\xi)\le C(1+\xi)^{-1}$ we can write
	\[
	\abs*{\int_{\R_+}f'(x)F(Rx)dx}\le \int_{\R_+}\frac{C\abs{b'(x)}}{1+Rx}dx,
	\]
	which is at most $\frac{C\normH{1}{b}}{\sqrt{R}}$. 
	This completes the proof.
\end{proof}
\begin{proof}[Proof of Theorem \ref{1:KS_estimate}]
Recall that by $F_{R, b}$ and $F_{\mathcal{N}}$ we denoted cumulative distribution functions of additive functionals $\overline{S_b^R}$ and standard Gaussian $\mathcal{N}(0, 1)$ respectively.
	Using the Feller smoothing estimate (see \cite[p.~538]{feller1966}), 
	Theorem \ref{1:BO_bessel} and Proposition \ref{6:lap_det} we have for any $T>0$
	\begin{multline}\label{5_eq:th_2}
		\sup_{x}\abs{F_{R, b} - F_{\mathcal{N}}} \le 
		\frac{24}{\sqrt{2\pi^3}T} +\\+ \frac{1}{\pi}\int_{-T}^T
		\frac{1}{\abs{k}}\abs*{e^{ikR\hat{b}(0) - ik\frac{\nu}{2}b(0)-ik
				\mathbb{E}_{J_\nu}S^R_b}Q_R^{\mathcal{B}}(kb) - 1}dk.
	\end{multline}
	By the second statement of Theorem \ref{1:BO_bessel} and Lemma 
	\ref{5:exp_conv} the expression under the integral may be 
	estimated by the following expression for some $C>0$, depending 
	only on $\normN{b_+}$ and $L(b)$
	\[
	\frac{C}{\sqrt{R}}(1 + \abs{k}^2)\exp\left(\frac{C\abs{k}}{\sqrt{R}}(1 + 
	\abs{k}^2)e^{C\abs{k}}\right).
	\]
	Observe that if $\abs{k} \le C_1\ln R$, where $C_1C <1/2$, then there 
	exists a constant $\widetilde{C}$ such that for any $R\ge 1$ the 
	expression above is at most
	\[
	\frac{\widetilde{C}(1+ (\ln R)^2)}{\sqrt{R}}.
	\]
	Therefore if we choose $T = C_1\ln R$ then the integral in 
	\eqref{5_eq:th_2} is at most
	\[
	\frac{2\widetilde{C}C_1\ln R(1+(\ln R)^2)}{\sqrt{R}}.
	\]
	And the statement of the theorem follows.
\end{proof}

\end{document}